\documentclass[a4paper]{article}

\usepackage[all]{xy}
\usepackage{amssymb,amsmath}
\usepackage{latexsym}
\usepackage[only,heavycircles]{stmaryrd}
\usepackage{color}
\usepackage{titlesec}
\usepackage[pointlessenum]
{paralist}
\usepackage[latin1]{inputenc}
\usepackage{comment}

\renewcommand{\H}{\operatorname{H}}
\DeclareMathOperator{\h}{\it h}
\newcommand{\C}{\mathbf{C}}
\newcommand{\N}{\mathbf{N}}
\DeclareMathOperator{\NS}{NS}
\DeclareMathOperator{\shhom}{\mathcal{H}\kern -.3ex\mathit{om}}
\DeclareMathOperator{\shann}{\mathcal{A}\kern -.05ex\mathit{nn}}
\DeclareMathOperator{\cliff}{Cliff}
\DeclareMathOperator{\Pic}{Pic}
\DeclareMathOperator{\Proj}{Proj}
\DeclareMathOperator{\Spec}{Spec}
\DeclareMathOperator{\im}{im}
\DeclareMathOperator{\Aut}{Aut}

\newcommand{\cond}[1]{\mathcal{C}_{#1}}
\newcommand{\curves}{\mathrm{Curves}}
\DeclareMathOperator{\ES}{ES}

\renewcommand{\subsetneq}{\varsubsetneq}

\newcommand{\ie}{i.e.\ }

\newcounter{memorise}
\def\Hom{\mathrm{Hom}}

\def\Z{\mathbf{Z}}

\def\P{\mathbf{P}}

\def\O{\mathcal{O}}

\def\cH{\mathcal{H}}
\def\X{\mathcal{X}}

\def\cC{\mathcal{C}}
\def\cD{\mathcal{D}}

\def\V{{V}}

\def\Proj{\mathrm{Proj}}
\def\separation{\medskip}

\renewcommand{\epsilon}{\varepsilon}
\newlength{\theorempostskipamount}
\newenvironment{theorem}[1][]
{\paragraph{\ Theorem.} {\normalfont #1} \it}
{\vspace{\the\theorempostskipamount}}
\newenvironment{lemma}[1][]
{\paragraph{\ Lemma.} {\normalfont #1} \it}
{\vspace{\the\theorempostskipamount}}
\newenvironment{proposition}[1][]
{\paragraph{\ Proposition.} {\normalfont #1} \it}
{\vspace{\the\theorempostskipamount}}
\newenvironment{definition}[1][]
{\paragraph{\ Definition.} {\normalfont #1} \it}
{\vspace{\the\theorempostskipamount}}
\newenvironment{corollary}[1][]
{\paragraph{\ Corollary.} {\normalfont #1} \it}
{\vspace{\the\theorempostskipamount}}

\newenvironment{problem}[1][]
{\paragraph{\ Problem.} {\normalfont #1} \it}
{\vspace{\the\theorempostskipamount}}
\newenvironment{conjecture}[1][]
{\paragraph{\ Conjecture.} {\normalfont #1} \it}
{\vspace{\the\theorempostskipamount}}

\newenvironment{remark}[1][]
{\paragraph{\ Remark.} {\normalfont #1}}
{\vspace{\the\theorempostskipamount}}

\newenvironment{example}[1][]
{\paragraph{\ Example.} {\normalfont #1}}
{\vspace{\the\theorempostskipamount}}
\newenvironment{examples}[1][]
{\paragraph{\ Examples.} {\normalfont #1}}
{\vspace{\the\theorempostskipamount}}

\newenvironment{warning}[1][]
{\paragraph{\ Warning.} {\normalfont #1}}
{\vspace{\the\theorempostskipamount}}

\newenvironment{proof}[1][Proof]{\noindent \textit{#1.~}}
{\hfill $\Box$ \\ }

\newcommand{\qed}{\hfill  $\Box$\separation}

\makeatletter
\def\@removefromreset#1#2{\let\@tempb\@elt
   \def\@tempa#1{@&#1}\expandafter\let\csname @*#1*\endcsname\@tempa
   \def\@elt##1{\expandafter\ifx\csname @*##1*\endcsname\@tempa\else
         \noexpand\@elt{##1}\fi}%
   \expandafter\edef\csname cl@#2\endcsname{\csname cl@#2\endcsname}%
   \let\@elt\@tempb
   \expandafter\let\csname @*#1*\endcsname\@undefined}

\@addtoreset{equation}{paragraph}
\@addtoreset{paragraph}{section}

\@addtoreset{footnote}{section}

\@removefromreset{paragraph}{subsubsection}
\@removefromreset{paragraph}{subsection}
\makeatother

\titleformat{\section}[hang]{\normalfont\Large\bfseries}{}{0cm}%
{\thesection \  --\ }

\titleformat{\subsection}[hang]{\normalfont\large\bfseries}{}{0cm}%
{\thesubsection\ -- \,}

\titleformat{\subsubsection}[hang]{\normalfont\bfseries}{}{0cm}{}

\renewcommand{\theparagraph}{(\arabic{section}.\arabic{paragraph})}
\titleformat{\paragraph}[runin]{\normalfont\bfseries}
{\theparagraph}{0cm}{}%[\kern -1em]
\titlespacing{\paragraph}{0cm}%left margin
{2.75ex plus 1ex minus .2ex}%beforesep
{.6em}%aftersep

\renewcommand{\thesubparagraph}
{(\arabic{section}.\arabic{paragraph}.\arabic{subparagraph})}
\titleformat{\subparagraph}[runin]{\normalfont}
{\thesubparagraph}{0cm}{}%[\kern -1em]
\titlespacing{\subparagraph}{0cm}%left margin
{0mm}%beforesep
{.3em}%aftersep

\makeatletter
\newcommand{\percentfill}{%
\leavevmode \cleaders \hb@xt@ .80em{\hss \%\hss }\hfill \kern \z@
}
\makeatother

\specialcomment{complimentary-computation}
{\medskip
\noindent\percentfill
\medskip\newline\indent}
{\par \medskip
\noindent\percentfill}

\newenvironment{introduction}
{\titleformat{\section}[hang]{\normalfont\Large\bfseries}{}{0cm}{}
\renewcommand{\theparagraph}{(\Alph{paragraph})}
\renewcommand{\thesubparagraph}{(\Alph{paragraph}.\arabic{subparagraph})}

\section*{Introduction}}
{}

\newenvironment{closing}
{\titleformat{\section}[hang]{\normalfont\large\bfseries}{}{0cm}{}
\setlength{\itemsep}{0mm}
\small}
{}

\makeatletter
\renewcommand\@maketitle{%
  \newpage
  \begin{center}%
  \let \footnote \thanks
    {\Large \bf \@title \par}%
    \vskip 1em%
    {\large
      \begin{tabular}[t]{c}%
        \@author
      \end{tabular}\par}%
  \end{center}%
  \par
  \vskip 1.5em}
\makeatother

\begin{document}
\renewcommand{\O}{\mathcal{O}}
\excludecomment{complimentary-computation}

\title{Equigeneric and equisingular families of curves on surfaces}
\author{T. %homas
Dedieu -- E. %doardo
Sernesi}
\date{}
\maketitle

\begin{abstract}
We investigate the following question:
let $C$ be an integral curve contained in a smooth complex algebraic
surface $X$; 
is it possible to deform $C$ in $X$ into a nodal curve 
while preserving its geometric genus?

We affirmatively answer it in most cases when $X$ is a Del Pezzo
or Hirzebruch surface (this is due to Arbarello and Cornalba, Zariski,
and Harris), and in some cases when $X$ is a $K3$ surface.
Partial results are given for all surfaces with numerically trivial
canonical class.
We also give various examples for which the answer is negative.
%
% \vskip .3cm
% \noindent
% MSC: 14H10 (primary), %Curves; Families, moduli (algebraic)
% 14H20, %Curves; Singularities, local rings
% 14B07 %Local theory; Deformations of singularities
% (secondary)\\[.1cm]
% Keywords: families of singular curves on algebraic surfaces;
% equigeneric and equisingular deformations;
% nodal curves.
\end{abstract}

\vskip 1cm
\hfill
\begin{minipage}[t]{8.5cm}
\small
Lascia lente le briglie del tuo ippogrifo, o Astolfo,\\
e sfrena il tuo volo dove pi\`u ferve l'opera dell'uomo.\\
Per\`o non ingannarmi con false immagini\\
ma lascia che io veda la verit\`a\\
e possa poi toccare il giusto.

\smallskip
%\hspace{1cm}
\hfill
\parbox[t]{7cm}
{---\,Banco del mutuo soccorso, 
% versi ispirati all'Orlando furioso
freely inspired by \emph{Orlando furioso}}
\end{minipage}

\begin{introduction}
Historically, the study of families of nodal irreducible plane
curves 
(the so--called \emph{Severi varieties}, named after 
\cite{Sev})
was motivated by the fact that every smooth projective curve is
birational to such a plane curve,
and that plane curves should be easier to study since they are
divisors.
One can of course consider similar families of curves in any smooth
algebraic surface and, as it has turned out, their study is rewarding
whether one is interested in surfaces or in curves.

Let $X$ be a smooth algebraic surface, and $\xi$ an element of its 
Néron--Severi group. For $\delta \in \Z _{\geq 0}$, we denote 
by $\V ^{\xi, \delta}$ the family of integral curves in $X$ of
class $\xi$, whose singular locus consists of exactly $\delta$
nodes
(\ie $\delta$ ordinary double points;
we call such curves \emph{nodal}, or \emph{$\delta$--nodal}).
These families are quite convenient to work with, being fairly 
well--understood from a deformation--theoretic point of view.
For instance, when the canonical class $K_X$ is non--positive this
enables one to show that they are smooth of the expected dimension in
the usual cases
(when $K_X$ is positive however, they tend to behave more wildly, 
see, e.g., \cite{CC99,ChS97}).
Moreover, they have been given a functorial definition 
in \cite{jW74b} (see also \cite[\S 4.7.2]{eS06}).

Yet, there is no definitive reason why one should restrict one's
attention to curves having this particular kind of singularities
(even when $X$ is the projective plane),
and it seems much more natural from a modular point of view to
consider the families $\V ^\xi _g$, $g \in \Z _{\geq 0}$, of integral
curves in $X$ of class $\xi$ that have geometric genus $g$
(\ie the normalizations of which have genus $g$).
We call these families \emph{equigeneric}.
These objects have however various drawbacks, for instance their
definition only makes sense set--theoretically, and 
accordingly there is no such thing as a local equigeneric deformation
functor (\ie one that would describe equigeneric deformations over
an Artinian base).

It is a fact that every irreducible equigeneric family $V$ of curves in
$X$ contains a Zariski open subset, all members of which have the same
kind of singularities (families enjoying the latter property are called
\emph{equisingular}), 
and these singularities determine via their deformation theory
the codimension $V$ is expected to have in the universal
family of all class $\xi$ curves in $X$.
This expected codimension is the lowest possible when the general
member of $V$ is nodal 
(in such a case, the expected codimension equals the number of nodes,
which itself equals the difference between the arithmetic and
geometric genera of members of $V$),
so that it makes sense to consider the following.

\begin{problem}
\label{Pb:main}
Let $C$ be an integral curve in $X$. %having geometric genus $g$.
Is it possible to deform $C$ in $X$ into a nodal curve 
% having the same geometric genus $g$?
while preserving its geometric genus?
\end{problem}

One may rephrase this as follows: 
let $\xi$ be the class of $C$ in $\NS(X)$,
$p_a(\xi)$ the arithmetic genus of curves having class $\xi$,
$g$ the geometric genus of $C$,
and $\delta = p_a(\xi) - g$;
is $\V ^\xi _g$ contained in the Zariski closure of $\V
^{\xi,\delta}$?
Observe that whenever the answer is affirmative, 
the Severi varieties $\V ^{\xi,\delta}$ provide a consistent way of 
understanding the equigeneric families $\V ^\xi _g$.

In any event, it is a natural question to ask what kind of
singularities does the general member of a given family $\V ^\xi _g$ 
have
(besides, this question is important for enumerative geometry,
see \cite{beauville-counting,FGS99,KS13}). 
Closely related to this is the problem of determining whether a given
equisingular family has the expected dimension. The actual dimension
is always greater or equal to the expected dimension, and whenever
they differ the family is said to be \emph{superabundant}.

In this text, we provide an answer to various instances of 
Problem~\ref{Pb:main}.
Some of these answers are not new, see below for
details and proper attributions.

\begin{theorem}\ 
\label{T:main}

\smallskip \subparagraph{}
{\normalfont (Arbarello--Cornalba \cite{AC80,AC81}, Zariski \cite{oZ82})}
\label{T:P^2}
Let $X=\P^2$ and $L=\O _{\P^2} (1) \in \Pic X = \NS (X)$.
For integers $n \geq 1$ and $0 \leq g \leq p_a(nL)$,
the general element of every irreducible component of $\V ^{nL} _g$
is a nodal curve.

\smallskip \subparagraph{}
{\normalfont (Harris \cite{jH86})}
\label{T:F_n}
Let $X$ be a degree $d$ Hirzebruch surface.
For every effective class $L \in \Pic X = \NS (X)$ and integer
$0 \le g \le p_a(L)$, 
the general member of every irreducible component of
$V ^L_{g}$ is a nodal curve.

\smallskip \subparagraph{}
{\normalfont (Harris \cite{jH86})}
\label{T:DP}
Let $X$ be a degree $d$ Del Pezzo surface, and 
$K_X \in \Pic X = \NS (X)$ its canonical class.
For integers $n\geq 1$ and $0 \leq g \leq p_a(-nK_X)$, 
the general element of every irreducible component of 
$\V ^{-nK_X} _{g}$ is nodal unless $dn \leq 3$
(it is at any rate immersed 
unless $d=n=1$ and $g=0$).

\smallskip \subparagraph{}
\label{T:K3}
Let $X$ be a very general algebraic $K3$ surface,
$L$ the positive generator of $\Pic X = \NS (X)$,
and write $L^2 = 2p-2$.
For $p/2 < g \leq p_a(L)=p$, 
the general element of every irreducible component of
$V ^L _g$ is nodal.

For integers $k \geq 1$ and $0 < g \leq p_a(kL)$, 
the general element of every irreducible component of
$\V ^{kL} _g$ is immersed;
if its normalization is non--trigonal\,\footnote
{When $k=1$, \cite{CK14} provides a sufficient
condition for a general element of $V^L _g$ to have a non--trigonal
normalization (see Corollary~\ref{coro:CK-nontrigonal}).},
then it is actually nodal.%}

\smallskip \subparagraph{}
\label{T:Enriques}
Let $X$ be an Enriques surface, and $L \in \Pic X = \NS(X)$ an
effective class.
For $3 \leq g \leq p_a(L)$, if $[C]\in \V ^L _{g}$ has
a non--hyperelliptic normalization $\bar{C}$,
then the general element of every component of $\V ^L _{g}$ 
containing $C$ is immersed. If moreover
$\bar C$ has Clifford index $\ge 5$, 
then $C$ is nodal.

\smallskip \subparagraph{}
\label{T:Abelian}
Let $X$ be an Abelian surface and $\xi \in \NS (X)$.
For $2 < g \leq p_a(\xi)$,
the general element of every irreducible component of
$\V ^{\xi} _g$ is immersed;
if its normalization is non--trigonal, then
it is actually nodal.
\end{theorem}

(Here a curve is said to be \emph{immersed} if the differential of its
normalization morphism is everywhere injective.)

In all cases within the above Theorem~\ref{T:main}, the corresponding 
Severi varieties
$\V ^{\xi, p_a(\xi)-g}$ are smooth and of the expected dimension
(if non empty;
non--emptiness is also known, except for Enriques and 
Abelian\,\footnote
{After the present text was completed, 
Knutsen, Lelli-Chiesa and Mongardi (arXiv:1503.04465)
proved the non-emptiness of $V^{\xi,p_a(\xi)-g}$ for
$\xi$ the numerical class of a polarization of type $(1,n)$ on an
Abelian surface, 
and $2 \leq g \leq p_a(\xi)$.}
surfaces).
In addition, their irreducibility has been proven in the following
cases: 
when $X$ is the projective plane \cite{jH86,harris-morrison},
when $X$ is a Hirzebruch surface \cite{tyomkin},
and when $X$ is a Del Pezzo surface, $g=0$, and $(d,n) \neq (1,1)$
\cite{testa};
when $X$ is a $K3$ surface, only a particular case is known
\cite{CD}.
These irreducibility properties transfer to the corresponding
equigeneric families when Problem~\ref{Pb:main} admits a positive
answer.

For surfaces with trivial canonical class one can formulate the
following conjecture, which Theorem~\ref{T:main} only partly solves.

\begin{conjecture}
\label{C:K3}
Let $X$ be a $K3$ (resp. Abelian) surface, 
and $\xi \in \NS (X)$.
For $g>0$ (resp. $g>2$) 
the general element of every irreducible component of
$\V ^{\xi} _g$ is nodal.
\end{conjecture}

Note however that Problem~\ref{Pb:main} does not always have a
positive solution.
This happens for instance when $X$ is a $K3$ (resp. Abelian) surface
and $g=0$ (resp. $g=2$); the latter case is however somewhat
exceptional, since the corresponding equigeneric families are
$0$--dimensional (see subsection~\ref{s:app-triv} for further
discussion). 
We give other instances, hopefully less exceptional,
of Problem~\ref{Pb:main} having a negative
solution in Section~\ref{S:museum}.
This comes with various examples, some of them new, of equigeneric and
equisingular families having superabundant behaviour.

 Surfaces of general type are missing from our analysis, as  their Severi varieties are notably not well-behaved and, especially, not keen to be studied using the techniques of the present text.  For information about this case one may consult \cite{CC99,ChS97}.

\bigskip
Problem~\ref{Pb:main} was first studied (and solved) for the
projective plane 
in the (19)80s by Arbarello and Cornalba \cite{AC80,AC81}
(see also \cite[Chap.~XXI \S\S8--10]{ACGII} for a unified treatment in
english),
and Zariski \cite{oZ82},
with different approaches.
The latter considers curves in surfaces as divisors and studies
the deformations of their equations
(we call this the \emph{Cartesian} point of view),
while the former see them as images of maps from smooth curves
(we call this the \emph{parametric} point of view).
Harris generalized this result using the Cartesian
theory in \cite{jH86},
thus obtaining as particular cases parts \ref{T:F_n} and \ref{T:DP} of
the above theorem.
There is however a subtle flaw in this text  (\cite[Prop.~2.1]{jH86})
 which  has been subsequently worked around using the parametric theory in
  \cite{harris-morrison}.
Apparently it had not been spotted before;
we   analyze it    in detail in subsection~\ref{s:pull-back}.

%The aforementioned flaw has disseminated in
%the literature, and it seems we are the first to spot it.

%and chose to give a unified account on Theorem~\ref{T:main}.
Note also that \cite[Lemma~3.1]{chen99} states 
Conjecture~\ref{C:K3} for $K3$ surfaces as a result,
but the proof reproduces the incomplete argument of
\cite[Prop.~2.1]{jH86};
unfortunately, in this case the parametric approach does not provide a
full proof either.
We also point out that the result of Conjecture~\ref{C:K3} for $K3$
surfaces is used in 
\cite[proof of Thm.~3.5]{bryan-leung}; the weaker 
Theorem~\ref{T:K3} should however be enough for this proof, see
\cite{FGS99,beauville-counting}.

Eventually let us mention that the recent \cite{KS13} by Kleiman and Shende provides an answer
to Problem~\ref{Pb:main} 
for rational surfaces under various conditions. They use the cartesian approach, while in the Appendix Tyomkin reproves the same results using the parametric approach. 

We need arguments from both the parametric and Cartesian approaches here.
The core of the parametric theory in the present text is
Theorem~\ref{T:nocusp}, 
which is essentially due to Arbarello--Cornalba, Harris, and
Harris--Morrison.
Except for its part \ref{T:K3}, Theorem~\ref{T:main} is a more or less
direct corollary of Theorem~\ref{T:nocusp};
parts \ref{T:Enriques} and \ref{T:Abelian}, which to the best of our
knowledge appear here for the first time\,\footnote
{Theorem~\ref{T:Abelian} has later on been used by
Knutsen, Lelli-Chiesa and Mongardi (arXiv:1503.04465)
to prove Conjecture~\ref{C:K3} for $X$ an Abelian surfaces and $\xi$
the numerical class of a polarization of type $(1,n)$.},
still require additional arguments from a different nature, admittedly
not new either (see subsection~\ref{s:app-triv}).

The parametric approach is more modern in spirit, and arguably more
agile, but although it enables one to give a full solution to
Problem~\ref{Pb:main} for minimal rational surfaces,
it does not provide a fully satisfactory
way of controlling equisingular deformations of curves;
somehow, it requires too much positivity of $-K_X$
(see, e.g., Remark~\ref{R:toostrong}), which explains why
Theorem~\ref{T:main} is not optimal in view of Conjecture~\ref{C:K3}.
For $K3$ surfaces, Theorem~\ref{T:K3},
which is  our main original contribution to the subject,
is beyond what is possible today with
the mere parametric approach;
we obtain it along the Cartesian approach, 
with the new tackle of formulating it in terms of general divisors on
singular curves (see section~\S\ref{s:gnlzddivs}),
and with the help of additional results from
Brill--Noether theory.
This is yet not a definitive answer either, and we believe finer
arguments are required in order to fully 
understand the subtleties of the question.

\bigskip
The organization of the paper is as follows.
In Section~\ref{S:background}, we define the abstract
notions of equigeneric and equisingular families of curves and specify
our setup.
In Section~\ref{S:parametric} we recall the relevant facts from the
parametric deformation theory, which culminate in the already
mentioned Theorem~\ref{T:nocusp}.
Section~\ref{S:cartesian} is devoted to Cartesian deformation theory,
which involves the so-called equisingular and adjoint ideals of an
integral curve with planar singularities.
% we introduce the new tackle of considering these ideals as generalized
% divisors on the corresponding curves, a key point in the proof of
% Theorem~\ref{T:K3}. 
In Section~\ref{S:apps} we apply the results of the two former
sections in order to prove Theorem~\ref{T:main},
and in Section~\ref{S:museum} we gather examples
in which the situation is not the naively expected one.

\bigskip \noindent
\textbf{Acknowledgements.}
We learned much from our reading of \cite{DH88} and
\cite[pp.~105--117]{harris-morrison}.
We also had the pleasure of helpful and motivating discussions
with C.~Ciliberto, L.~Ein, and C.~Galati,
and are grateful to F.~Flamini for his careful reading of this text.
Our special thanks go to A.~L.~Knutsen who in particular showed us the
crucial Example~\ref{Ex:jacobian} at the right time.
Finally, we thank X.~Chen for having kindly answered our questions
about his work \cite{chen99}.

This project profited of various visits of the authors one to another,
which have been made possible by the research group GRIFGA, in
collaboration between CNRS and INdAM.
T.D. was partially supported by French ANR projects CLASS and MACK.  E.S. is a member of GNSAGA--INdAM and was partially supported
by the project MIUR-PRIN \emph{Geometria delle variet\`a algebriche.}
\end{introduction}

\section{Equigeneric and equisingular families of curves}
\label{S:background}

We work over the field $\C$ of complex numbers.

\subsection{General definitions}

While the definition of equigenericity is rather straightforward, that
of equisingularity is much more subtle, and requires some care.
The definition given here is taken from Teissier \cite{bT77,bT80}, who
slightly modified the one originally introduced by Zariski (see
\cite[\S 5.12.2]{bT77} for a comment on this).
The two versions are anyway equivalent in our setting 
(explicited in subsection~\ref{s:setting})
by \cite[II, Thm.~5.3.1]{bT80}.
We invite the interested reader to take a look at \cite{DH88} as well.

Let $p: \mathcal{C} \to Y$ be a flat family of reduced %projective
curves,
where $Y$ is any separated scheme.
\begin{definition}
\label{d:equigen}
The family $p: \mathcal{C} \to Y$ is equigeneric if
\begin{inparaenum}
\renewcommand{\theenumi}{\normalfont (\roman{enumi})}
\item $Y$ is reduced,
\item the locus of singular points of fibres is proper over $Y$,
and \item\label{c:delta-inv}
 the sum of the $\delta$-invariants of the singular points of
the fibre $C_y$ is a constant function on $y \in Y$.
\end{inparaenum}
\end{definition}

When $p$ is proper, condition~\ref{c:delta-inv} above is equivalent to the
geometric genus of the fibres being constant on $Y$.

\begin{definition}
\label{d:equising}
The family $p: \mathcal{C} \to Y$ is equisingular if
there exist
\begin{inparaenum}[\normalfont (a)]
\item disjoint sections $\sigma_1, \ldots, \sigma_n$ of
$p$, the union of whose images contains the locus of singular
points of the fibres,
and \item a proper and birational morphism $\epsilon: \bar{\cal C} \to
{\cal C}$,
\end{inparaenum}
such that
\begin{inparaenum}[\normalfont (i)]
\item the composition ${\bar p}:= p\circ\epsilon: \bar{\cal C}
  \to Y$ is flat,
\item for every $y\in Y$, the induced morphism $\epsilon_y:\bar C_y
  \to C_y$ is a resolution of singularities (here $\bar C_y$ and $C_y$
  are the respective fibres of $\bar p$ and $p$ over $y$),
and \item for $i=1,\ldots,n$, the induced morphism $\bar p:
\epsilon^{-1}(\sigma_i(Y)) \to Y$ is locally (on
$\epsilon^{-1}(\sigma_i(Y))$) trivial.
\end{inparaenum}
\end{definition}

In Definition \ref{d:equigen}, the reducedness assumption on the base
is an illustration of the fact that equigenericity cannot be
functorially defined, unlike equisingularity.
%(see \S\ref{s:subsysts})
%
The following result of Zariski, Teisser, Diaz--Harris, provides a
more intuitive interpretation of equisingularity.
Two germs of isolated planar curve singularities 
$(C_1,0) \subset (\C^2,0)$ and $(C_2,0) \subset (\C^2,0)$
are said to be topologically equivalent if there exists a
homeomorphism $(\C^2,0) \to (\C^2,0)$ mapping $(C_1,0)$ to $(C_2,0)$
(cf. \cite[I.3.4]{GLSb}).
The corresponding equivalence classes are called 
\emph{topological types}.
\begin{theorem}
  [{\cite[II, Thm.~5.3.1]{bT80}, \cite[Prop.~3.32]{DH88}}]
Let $p: \mathcal{C} \to Y$ be a flat family of reduced curves on a
smooth surface $X$, \ie  ${\cal C} \subset X \times Y$, and $p$ is
induced by the second projection. We assume that ${\cal C}$ and
$Y$ are reduced separated schemes of finite type. Let $\Sigma
\subset \cal C$ be the locus of singular points of fibres of $p$.
If $\Sigma$ is proper over $Y$   the following two conditions are
equivalent:
\begin{compactenum}[\normalfont (i)]
\item 
the family $p: \mathcal{C} \to Y$ is locally equisingular in the
%\'etale
analytic topology;
\item 
for each topological type of isolated planar curve singularity, 
all fibres over closed
points of $Y$ have the same number of singularities of that
topological type.
\end{compactenum}
\end{theorem}

One then has the following result, 
%rather obvious from an intuitive point of view, and 
often used without any mention in the literature.
It is an application of the generic smoothness theorem.
%comment
% see also the discussion at the beginning of the proof of Prp.2,
% p.12 in \cite{KS}
\begin{proposition}[{\cite[II, 4.2]{bT80}}]
\label{P:gen-equising} 
Let $p: \cC \to Y$ be an equigeneric family of %projective
reduced %irreducible
curves. There exists a dense Zariski-open subset
$U\subset Y$ such that the restriction $\cC\times_YU\to U$ is
equisingular.
\end{proposition}

The latter result implies the existence, for any flat family $p:
\cC \to Y$ of reduced curves on a smooth surface $X$, with $Y$
reduced separated and of finite type, of an equisingular
stratification of $Y$ in the Zariski %analytic
topology. Indeed, the
geometric genus of the fibres being a lower semi-continuous
function on $Y$ (see e.g. \cite[\S2]{DH88}), our family restricts
to an equigeneric one above a %dense
Zariski-open subset of $Y$, to which
we can apply Proposition \ref{P:gen-equising}.

Eventually we need the following result of Teissier, which shows that
equigenericity can be interpreted in terms of the existence of a
simultaneous resolution of singularities.
\begin{theorem}
[{\cite[I, Thm.~1.3.2]{bT80}}]
%{\normalfont \cite[I, Thm.~1]{bT80}}
\label{T:simult}
Let $p: \mathcal{C} \to Y$ be a flat family of reduced %projective
curves, where $\cal C$ and $Y$ are reduced separated schemes of finite
type.
If $Y$ is normal,
then the following two conditions are
equivalent:
\begin{compactenum}[\normalfont (i)]
\item
the family $p: \mathcal{C} \to Y$ is equigeneric;
\item
there exists a proper and birational morphism $\epsilon: \bar{\cal C} \to
{\cal C}$, 
such that $\bar {p} = p \circ \epsilon$ is flat, 
and for every $y \in Y$, the induced morphism
$\bar C_y \to C_y$ is a resolution of singularities
of the fibre $C_y = p ^{-1}(y)$.
\end{compactenum}

In addition, whenever it exists, the simultaneous resolution $\epsilon$
is necessarily the normalization of $\cal C$.
\end{theorem}

\subsection{Superficial setting}
\label{s:setting}

We now introduce our set-up for the remaining of this paper.

\paragraph{}
Unless explicit mention to the contrary, $X$ shall design a
nonsingular projective connected   algebraic
surface. Given an element $\xi \in \NS (X)$ of the Néron-Severi group
of $X$ we let
\[
\Pic ^\xi (X) := \{ L\in \Pic(X)\
| \ 
\text{$L$ has class $\xi$}\}.
\]
The  Hilbert scheme of effective divisors of $X$ having class $\xi$,
which we denote by $\curves^\xi_X$, is fibered  over
$\Pic^\xi(X)$
\[
\curves ^\xi_X \longrightarrow \Pic ^\xi(X)
\]
with fibres linear systems. We write $p_a(\xi)$ for the common arithmetic genus of all members of
$\curves^\xi_X$.  In case $q(X):=\h^1(X,\O_X)=0$, \ie $X$ is
\emph{regular}, 
$\curves^\xi_X$ is a disjoint union of finitely many linear
systems $|L|$, with $L$ varying in $\Pic^\xi(X)$.

\paragraph{}
For any given integer $\delta$ such that  $0 \le \delta \le
p_a(\xi)$ there is a well defined, possibly empty,  locally closed
\emph{subscheme} $V^{\xi,\delta}\subset \curves^\xi_X$,
whose geometric points parametrize reduced and irreducible curves
having exactly $\delta$ nodes and no other singularities. These
subschemes are defined functorially in a well known way
\cite{jW74b} and will be called \emph{Severi varieties}.

More generally, given a reduced curve $C$ representing 
$\xi \in \NS(X)$,
there is a functorially defined subscheme 
$\ES(C) \subset \curves ^\xi _X$
whose geometric points parametrize those reduced curves that
have the same number of singularities as $C$ for every
equivalence class of planar curve singularity
\cite{jW74}.
The restriction to $\ES(C)$ of the universal family of curves over
$\curves ^\xi _X$ is the largest equisingular family
of curves on $X$ that contains $C$.

\paragraph{}
We will also consider, for any given integer  $g$ such
that $0 \le g \le p_a(\xi)$, the locally closed \emph{subset}
$V^\xi_g\subset \curves^\xi_X$  whose geometric points
parametrize reduced and irreducible curves $C$  having geometric
genus $g$,  \ie  such that their normalization has genus $g$. When
$\delta = p_a(\xi)-g$ we have  $V^{\xi,\delta}\subset V^\xi_g$.

There is also, for each $L \in \Pic^\xi(X)$, a
subscheme
$ %\[
V^\delta_L = V^{\xi,\delta} \cap |L|
$ %\]
of $|L|$, and a locally closed subset
$ %\[
V_{L,g} = V^\xi_g \cap |L|
$. %\]
These are the natural objects to consider when $X$ is regular.

\section{A parametric approach}
\label{S:parametric}

\subsection{The scheme of morphisms} \label{S:morph}

We briefly recall some facts from the deformation theory of maps with
fixed target, which will be needed later on.
Our main reference for this matter is \cite[\S 3.4]{eS06};
\cite[Chap.~XXI \S\S8--10]{ACGII} may also be useful.
We consider  a fixed
nonsingular projective $n$-dimensional variety $Y$. 

\paragraph{}
\label{p:modular}
We use the definition of \emph{modular family}, as given in
\cite[p. 171]{hart}.
For every $g \geq 0$, there is
a modular family $\pi _g : \mathcal{D}_g \to  S_g$ of smooth projective
connected curves of genus $g$ by
\cite[Thm.~26.4 and Thm.~27.2]{hart},
and $S_g$ has dimension $3g-3+a_g$, with $a_g$ the dimension of the
automorphism group of any genus $g$ curve.
Then, setting 
$M_g (Y)$ to be the relative $\Hom$ scheme
$\Hom ({\cal D}_g / S_g, Y \times S_g / S_g)$
and ${\cal D}_g(Y) := \cD_g \times _{S_g} M_g (Y)$,
there is a modular family of morphisms from nonsingular projective
connected curves   of genus $g\ge 0$ to $Y$
in the form of the commutative diagram
 \begin{equation}\label{E:mor1}
\xymatrix@=15pt{
{\cal D}_g(Y) \ar[r]^-{\Phi_g} \ar[d]
& Y \times M_g(Y) \ar[dl]
\\
M_g(Y)}
\end{equation}
% where $M_g(Y)$ is an algebraic    scheme,  
% $\mathcal{D} \to M_g(Y)$ is a flat family of smooth projective connected curves of genus $g$ and
% the commutative diagram 
which enjoys properties (a),(b),(c) of \cite[Def.\ p.171]{hart}
(note that here we declare two morphisms to be isomorphic if they are
equal).
% Diagram
% \eqref{E:mor1} can be constructed   as in \cite[Prop. II.1.13]{jK96},
% , replacing the family $f:X\longrightarrow S$ thereby considered with
% a modular family $f:\mathcal{C}\longrightarrow S$ of smooth projective
% curves of genus $g$, as constructed e.g. in \cite[Theorem
% 27.2]{hart}. 
Note that the scheme $M_g(Y)$ and diagram \eqref{E:mor1}
are unique only up to an étale base change;  
nevertheless, with an abuse of language we call $M_g(Y)$   \emph{the scheme of morphisms} from curves of genus $g$ to $Y$. % and we sometimes denote it by $M(\phi)$ to emphasize its dependency on $\phi$.

Let
\[
\phi: D \to Y
\]
be a morphism from a nonsingular connected projective  curve $D$ of
genus $g$
% birational onto its image 
and $[\phi]\in M_g(Y)$ a point parametrizing it. 
There is an exact sequence (\cite[Prop.~4.4.7]{eS06})
\begin{equation}\label{E:thilb1}
% \xymatrix{
% 0 \ar[r] & H^0(D,\phi^*T_Y) \ar[r]& T_{[\phi]}M_g(Y) \ar[r]&H^1(D,T_D) \ar[r]&H^1(D,\phi^*T_Y)}
0 \to  \H^0(D,\phi^*T_Y) 
\to T_{[\phi]} M_g(Y) 
\to \H^1(D,T_D) 
\to \H^1(D,\phi^*T_Y),
\end{equation}
and it follows from \cite[I.2.17.1]{jK96}
% from \cite[Prop.~II.1.13]{jK96}, 
% as in the proof of \cite[Thm.~II.1.14]{jK96}, 
that
\begin{equation}\label{E:thilb2}
-K_Y \cdot \phi_* D +(n-3)(1-g)+\dim\left(\mathrm{Aut}(D)\right)
\le \dim_{[\phi]}M_g(Y).
\end{equation}

\paragraph{}
We denote by  $\mathrm{Def}_{\phi/Y}$
the deformation functor of $\phi$ with fixed target $Y$, as
introduced in \cite[\S 3.4.2]{eS06}.  %It follows from property (b) that if $m_o\in M$ parametrizes $\phi$ then $\mathrm{Def}_{\phi/Y}$  is prorepresented by $\widehat{\O}_{M,m_o}$.
Recall  that $N_\phi$, 
the \emph{normal sheaf of $\phi$}, 
is the sheaf of $\O_{D}$-modules
defined by the exact sequence on $D$
\begin{equation}\label{E:mor2}
% \xymatrix{ 0 \ar[r]&T_{\bar{C}}\ar[r] & \phi^*T_X \ar[r]&
% N_\phi \ar[r] & 0 }
0 \to T_{D} \to   \phi^*T_Y \to
N_\phi \to   0.
\end{equation}
It controls the functor $\mathrm{Def}_{\phi/Y}$: one has
$\mathrm{Def}_{\phi/Y}(\C[\epsilon])= \H^0(D,N_\phi)$, and
$\H^1 (D, N_\phi)$ is an obstruction space for
${\rm Def}_{\phi/Y}$;    
in particular, if $R_\phi$ is the complete local algebra which
prorepresents $\mathrm{Def}_{\phi/Y}$ 
(\cite[Thm.~3.4.8]{eS06}),  we have
\[\chi(N_\phi) \le \dim(R_\phi)\le
h^0(N_\phi).
\]
Using the exact sequence \eqref{E:mor2}, one computes
\begin{equation}\label{E:mor2.1}
\chi(N_\phi) = 
\chi(\omega_{D}\otimes\phi^*\omega_Y^{-1}) 
= -K_Y \cdot \phi_* D +(n-3)(1-g),
\end{equation}
 hence
\begin{equation}\label{E:mor3}
-K_Y \cdot \phi_* D +(n-3)(1-g) \le \dim(R_\phi)\le
h^0(N_\phi)=\chi(N_\phi)+h^1(N_\phi).
\end{equation}
In particular, $R_\phi$ is smooth of dimension 
$-K_Y \cdot \phi_* D +(n-3)(1-g)$
  if and only if ${\rm H}^1(N_\phi)=0$.

\paragraph{}
In analyzing the possibilities here, one has to keep in mind that
$N_\phi$ can have torsion. In fact there is an exact sequence
of sheaves of $\O_{D}$-modules
\begin{equation}\label{E:mor4.0}
    % \xymatrix{
    % 0 \ar[r] & \cH_\phi \ar[r] & N_\phi \ar[r] &
    % \overline{N}_\phi \ar[r] & 0}
    0 \to \cH_\phi \to N_\phi \to
    \bar{N}_\phi \to 0,
\end{equation}
where $\cH_\phi$ is the torsion subsheaf of $N_\phi$, and
$\bar{N}_\phi$ is locally free. 
The torsion sheaf $\cH_\phi$ is supported on the ramification
divisor $Z$ of $\phi$, and it is zero if and only if $Z=0$.
Moreover, there is an exact sequence of locally free sheaves on
$D$
\begin{equation}
\label{E:mor4}
0 \to
T _{D} (Z) \to \phi ^* T_Y \to
{\bar N} _\phi
\to 0.
\end{equation}

\paragraph{}
The scheme $M_g(Y)$ and the functors $\mathrm{Def}_{\phi/Y}$ are related as follows. For each $[\phi]\in M_g(Y)$ we get by restriction a morphism from the  prorepresentable functor   
$h_{\widehat{\O}_{M_g(Y),[\phi]}}$ to  $\mathrm{Def}_{\phi/Y}$. Call $\rho_\phi$ this morphism. Its differential is described by the diagram:
\[
\xymatrix{
0 \ar[r] & H^0(D,\phi^*T_Y) \ar[d]\ar[r]& T_{[\phi]}M \ar[d]^-{d\rho_\phi}\ar[r]&H^1(D,T_D)\ar@{=}[d] \ar[r]&H^1(D,\phi^*T_Y)\ar@{=}[d] \\
0 \ar[r]& H^0(D,\phi^*T_Y)/H^0(D,T_D)\ar[r]& H^0(D,N_\phi) \ar[r]& H^1(D,T_D)\ar[r]&H^1(D,\phi^*T_Y)
}
\]
where the top row is the sequence \eqref{E:thilb1} and the second row is
deduced from the sequence \eqref{E:mor2}.  This diagram shows that
$d\rho_\phi$ is surjective with kernel $H^0(D,T_D)$, whose dimension
is equal to $\dim\left(\mathrm{Aut}(D)\right)$. In particular, if
$M_g(Y)$ is smooth at $[\phi]$, then $\mathrm{Def}_{\phi/Y}$ is smooth
as well and  
$\dim(R_\phi)= \dim_{[\phi]}(M_g(Y))-
\dim\left(\mathrm{Aut}(D)\right)$.  
This analysis is only relevant when $g=0,1$, because otherwise
$\rho_\phi$ is an isomorphism.

\subsection{Equigeneric families and 
schemes of morphisms}

In view of the superficial situation set up in 
subsection~\ref{s:setting},
we will often consider the case when $\phi$ 
is the morphism $\varphi: \bar C \to X$,
where $C$ is an integral curve in a smooth projective surface $X$,
and $\varphi$ is the composition of the normalization
$\nu: \bar C \to C$ with the inclusion $C \subset X$;
we may loosely refer to $\varphi$ as the normalization of
$C$.
We then have
\begin{equation*}
{\bar N} _\varphi \cong 
\varphi ^* \omega_X ^{-1} \otimes \omega _{\bar C} (-Z)
\end{equation*}
by the exact sequence \eqref{E:mor4}.
The embedded curve $C$ is said to be \emph{immersed} if the
ramification divisor $Z$ of $\varphi$ is zero;
in this case, we may also occasionally say that $C$ has no
(generalized) cusps.

The following result is based on a crucial observation by Arbarello
and Cornalba \cite[p. 26]{AC81}.

\begin{lemma}
\label{L:AC} %AC = Arbarello--Cornalba e assai curioso
% fu \label{L:comp2}
Let $B$ be a semi-normal\,\footnote
{we refer to \cite[\S I.7.2]{jK96} for background on
  this notion.}
connected scheme, $0 \in B$ a closed point, $\pi: \cD \to B$ a flat family of
smooth projective irreducible curves of genus $g$, and
\[ \xymatrix@=15pt{
\mathcal{D} \ar[d] _{\pi} \ar[r] ^(.35){\Phi}
& X \times B \ar[dl] ^{\mathrm{pr}_2} \\
B
} \]
a family of morphisms.
We call $D_0$ the fibre of $\cD$ over $0\in B$, 
$\phi_0: D_0 \to X$ the restriction of $\Phi$, which we assume to be
birational on its image,
and $\xi$ the class of $\phi_0 (D_0)$ in $\NS (X)$.

\subparagraph{}
The scheme-theoretic image $\Phi (\mathcal{D})$ is flat over $B$.

This implies that there are two classifying morphisms 
$p$ and $q$
from $B$ to $M_g(X)$ and $\curves_X ^{\xi}$
respectively, with differentials
\[
d\rho_{\phi_0} \circ dp_0 : T_{B,0} \to \H^0 (D_0, N _{\phi_0})
\quad \text{and} \quad 
dq_0 : T _{B,0} \to \H^0 (\phi_0(D_0), N _{\phi_0(D_0)/X}).
\]
%(if $g\le 1$ then $dp_0$ must be replaced by $d\rho_{\phi_0}\cdot dp_0$)

\subparagraph{}
The inverse image by $d\rho_{\phi_0} \circ dp_0$ of the torsion
$\H^0 (D_0, \mathcal{H}_{\phi_0}) \subset \H^0 (D_0, N _{\phi_0})$
is contained in the kernel of $dq_0$.
\end{lemma}

\begin{proof}  
The morphism $\varpi=\mathrm{pr}_1: \Phi (\cD) \to B$
is a well-defined family of
codimension $1$ algebraic cycles of $X$ in the sense of
\cite[I.3.11]{jK96}.
Since $B$ is semi-normal, it follows from \cite[I.3.23.2]{jK96}
that $\varpi$ is flat.

Given a non-zero section $\sigma \in H^0(D_0,N_{\phi_0})$,
the first order deformation of ${\phi_0}$ defined
by $\sigma$ can be described in the following way: consider an
affine open cover $\{U_i\}_{i\in I}$ of $C_0$, and for each $i
\in I$ consider a lifting $\theta_i\in \Gamma(U_i,{\phi_0^*}T_X)$
of the restriction   $\sigma_{|U_i}$. Each $\theta_i$ defines a
morphism
\[
\psi_i: U_i\times
\mathrm{Spec}(\C[\epsilon])\longrightarrow X
\]
extending ${\phi_0}_{|U_i}: U_i\to X$.  The morphisms
$\psi_i$ are then made compatible after gluing the trivial
deformations $U_i\times \mathrm{Spec}(\C[\epsilon])$ into
 the first order deformation of $D_0$ defined by the
coboundary $\partial(\sigma) \in H^1(C_0,T_{C_0})$
of the exact sequence \eqref{E:mor2}. In case
$\sigma \in \H^0(D_0,\cH_{\phi_0})$, everyone of the maps
$\psi_i$ is the trivial deformation  of $\sigma_{|U_i}$ over an
open subset. This implies that the corresponding first order
deformation of ${\phi_0}$ leaves the image fixed,
hence the vanishing of $dq_0 (\sigma)$.
\end{proof}

\begin{lemma}\label{L:comp1}
Let $m_0\in M_g(X)$ be  a general  point of an irreducible component
of $M_g(X)$, 
and $\phi_0: D_0 \to X$ the corresponding morphism. 
Assume that $\phi_0$ is birational onto its
image $C_0 := \phi_0(D_0)$, and that $[C_0]\in V^\xi_g$. Then
$[C_0]$ belongs to a unique irreducible component of $V^\xi_g$ and  
\[
\dim_{[C_0]} V^\xi_g
= \dim R_\varphi = 
\dim_{m_0} M_g(X) -
\dim (\operatorname{Aut} D_0 ),
\]
where $R_{\phi_0}$ is the complete local $\C$-algebra that
prorepresents $\mathrm{Def}_{{\phi_0}/X}$. 
\end{lemma}

\begin{proof}
 % We give the proof in the case $g \ge 2$;    then we need to prove
 % that $\dim_{[C_0]}(V^\xi_g)=\dim_{m_0}(M_g(X))$. 
Consider the reduced scheme $M_{red}:= M_g(X)_{red}$, 
and let $\tilde M$ be its semi-normalization.
Let  
\[ \xymatrix@=15pt{
% \bar{\mathcal{C}}_B \ar[d] \ar[r] ^(.35){\Phi_B}
% & X \times B \ar[dl] \\
% B
{\mathcal{D}} _{\tilde M} \ar[d] \ar[r] ^(.4){\Phi_{\tilde M}}
& X \times \tilde{M} \ar[dl] \\
\tilde{M}
} \]
be the pullback of the modular family %\eqref{E:morC}.
\eqref{E:mor1}.
Then we have a diagram
\[\xymatrix@=15pt{ 
\Phi_{\tilde{M}} ({\mathcal{D}} _{\tilde M}) \ar[d] _\pi \ar@{^(->}[r] 
& X\times \tilde{M} \ar[dl] \\
\tilde{M}
}\]
where $\Phi_{\tilde{M}} ({\mathcal{D}} _{\tilde M})$
is the scheme-theoretic image.  
The morphism $\pi$ is flat by Lemma~\ref{L:AC}, 
and therefore we have an induced functorial morphism
$\Psi: \tilde{M} \to V^\xi_g$.

 Suppose that $\Psi(m_1)=\Psi(m_0)=[C_0]$ for some $m_1 \in 
\tilde M$. Then
 $m_0$ and $m_1$ parametrize the same morphism up to an automorphism
 of the source $D_0$.
By  property (a) of  modular families, this implies that the fibres of
$\Psi$ have the same dimension as $\operatorname{Aut} D_0$,
and therefore that
\[
\dim_{m_0} M_g(X) - \dim (\operatorname{Aut} D_0)
%=\dim(M)=\dim(\overline{\Psi(M)}) 
\le \dim_{[C]} V ^\xi _g.
\]

On the other hand, consider the normalization map $\bar{V}^\xi_g \to
V^\xi_g$, and the pull-back to $\bar{V}^\xi_g$ of the universal family
of curves over $\curves^\xi_X$. 
It has a simultaneous resolution of
singularities $\bar {\mathcal {U}} \to \bar V ^\xi _g$
by Theorem~\ref{T:simult},
which comes with a family of morphisms
$N: \bar {\mathcal{U}} \to X \times \bar V ^\xi _g$ over
$\bar V ^\xi _g$. 
By property (c) of the modular family $M_g(X)$, this implies that 
%for every point $[C]_i$ in the preimage of $[C]$ in $\bar{V}^\xi_g$, 
there exist an 
  \'etale surjective $\eta:W\longrightarrow \bar{V}^\xi_g$
  and a morphism
$w:W \to M_g(X)$ such that 
$\bar {\mathcal{U}} _{W}
:=\bar {\mathcal{U}} \times _{\bar V ^\xi _g} W$
fits in the Cartesian diagram
\[ \xymatrix@=15pt{
\bar {\mathcal{U}} _{W} \ar[r] \ar[d] {\ar @{} [dr] |{\Box}}
& \cD_g \ar[d] \\
W \ar[r]^-w & M_g(X)
} \]
where the left vertical map is the pullback of $N$.
The map $W \to M_g(X)$ is generically injective because the universal
family of curves over $V^\xi_g$ is nowhere isotrivial.
Moreover, its image is transverse at every point $m$ (corresponding to
a morphism $\phi$) to the subvariety of $M_g (X)$  parametrizing
morphisms gotten by composing $\phi$ with an automorphism of its
source. 
This implies that, for $c_0\in \eta^{-1}([C_0])$ 
\[
\dim_{[C_0]}V^\xi_g 
= \dim_{c_0}W
\le \dim_{m_0} M_g(X)-\dim (\operatorname{Aut} D_0 ).
\]
It is then clear that $[C_0]$ belongs to a unique irreducible
component of $V^\xi_g$.
\end{proof}

\begin{remark}
\label{Ex:nonredM}
It can happen that 
% the differential of $\Psi$ is nowhere injective 
$M_g(X)$ is non-reduced.
For an example of such a situation, consider the pencil $|L|$ 
constructed in Example~\ref{Ex:cayleybacharach} below
(we use the notations introduced therein), 
and let $C\subset X$ be a general element of $\V_{L,9}$,
which is open and dense in $|L|$.
The curve $C$ has one ordinary cusp $s$ and no further singularity;
we let $s' \in \bar C$ be the unique ramification point of the
normalization $\varphi: \bar C \to X$.

One has $\chi(N_{\varphi})= -8+8=0$ whereas 
$\dim _{[\varphi]} M_9 (X) = 1$. 
The torsion part of $N _\varphi$ is the skyscraper sheaf 
$\underline \C _{s'}$, and accordingly $\h^0(\cH_\varphi)=1$.
One has $\h^0(\omega_{\bar{C}}\otimes \varphi^*\omega_X^{-1})=1$, 
and the unique divisor in 
$|\omega_{\bar{C}}\otimes \varphi^*\omega_X^{-1}|$
contains $s'$ (with multiplicity $4$), so that
\[
\h^0(\bar{N}_\varphi)=\h^0\left(\omega_{\bar{C}}\otimes
\varphi^*\omega_X^{-1}(-{s'})\right)=
\h^0(\omega_{\bar{C}}\otimes \varphi^*\omega_X^{-1})=1.
\]
We then deduce from the exact sequence \eqref{E:mor4.0} that
$\h^0(N_\varphi)=2$
and $\h^1(N_\varphi)=2$.
\end{remark}

\subsection{Conditions for the density of nodal 
(resp. immersed) curves}

The following result is essentially contained in 
\cite{jH86,harris-morrison}; the idea of condition \ref{c:3ample}
therein comes from \cite{AC80}.

\begin{theorem}\label{T:nocusp}
Let $V \subset V^\xi_g$ be an irreducible component and let $[C]
\in V$ be a general point, with normalization
$\varphi: \bar C \to X$.

\smallskip
\subparagraph{}
\label{T:nocuspI}
Assume that the following two conditions
are satisfied:
\begin{compactenum}
\renewcommand{\theenumi}{\normalfont (\alph{enumi})}
\item
\label{c:bpf}
$\omega_{\bar C}\otimes \varphi^*\omega_X^{-1}$ is globally 
generated;
\item
\label{c:dim}
$\mathrm{dim}(V) \ge h^0(\bar C,\omega_{\bar C}\otimes
\varphi^*\omega_X^{-1})$.
\setcounter{memorise}{\value{enumi}}
\end{compactenum}
Then $C$ %has no cusps,
is immersed, \ie  all its singularities consist of
(possibly non transverse) linear branches.

\smallskip
\subparagraph{}
\label{T:nocuspII}
Assume in addition that the following condition is satisfied:
\begin{compactenum}
\renewcommand{\theenumi}{\normalfont (\alph{enumi})}
\setcounter{enumi}{\value{memorise}}
\item
\label{c:3ample}
the line bundle $\omega_{\bar C}\otimes \varphi^*\omega_X^{-1}$
separates any (possibly infinitely near) $3$ points,
\ie 
\[h^0\big(\bar C,\omega_{\bar C}\otimes
\varphi^*\omega_X^{-1}(-A)\big)= h^0 \big(\bar C,\omega_{\bar C}\otimes
\varphi^*\omega_X^{-1}\big)-3
\]
for every effective divisor $A$ of
degree $3$ on $\bar C$.
\end{compactenum}
Then %$[C] \in V^{\xi,\delta}$. 
$C$ is nodal.
Equivalently $V \subset \overline{V^{\xi,\delta}}$, 
with $\delta=p_a(\xi)-g$.
\end{theorem}

\begin{proof}
For simplicity we give the proof in the case $g \ge 2$. 
Assume by contradiction that the curve $C$ has (generalized) cusps. 
This is equivalent to the fact that $Z \ne 0$,
where $Z\subset \bar C$ is the zero divisor of 
the differential of $\varphi$.
By generality, $[C]$ is a smooth point of $V$, so we may (and will)
assume without loss of generality that $V$ is smooth.
As in the proof of Lemma~\ref{L:comp1}, it follows from 
Theorem~\ref{T:simult} that there is a simultaneous resolution of
singularities
\[\xymatrix@=15pt{
\bar {\cC} \ar[dr] _{\bar \pi} \ar[r] ^\Phi
& \cC \ar[d] ^\pi \\
& V
}\]
of the universal family of curves over $V$.
This is a deformation of the morphism $\varphi$, so we have a 
characteristic morphism $p: V \to M_g(X)$. The differential
\[
dp_{[\varphi]} : T _{[C]} V \to \H^0 (\bar C, N _\varphi)
\]
is injective because to every tangent vector $\theta \in T _{[C]} V$
corresponds a non-trivial deformation of $C$.
On the other hand, it follows from Lemma~\ref{L:AC} that
the intersection $\operatorname{Im} dp_{[\varphi]} \cap 
\H^0 (\bar C, \mathcal {H} _{\varphi})$ is trivial.
Eventually, we thus have
\begin{equation*}
\dim V
= \dim T _{[C]} V
\leq h^0 (\bar C, \bar N _\varphi )
= \h^0(\omega_{\bar C}\otimes \varphi^*\omega_X^{-1}(-Z)).
\end{equation*}
By assumption \ref{c:bpf}, this implies $\dim V < 
\h^0(\omega_{\bar C}\otimes \varphi^*\omega_X^{-1})$, a
contradiction. 
This proves \ref{T:nocuspI}.

Assume next that \ref{c:3ample} is also satisfied
and, by contradiction, that $C$ is not nodal.
We shall show 
along the lines of \cite[pp.~97--98]{AC80}
that it is then possible to deform $C$ into curves with
milder singularities, which contradicts the generality of $C$ in $V$
and thus proves \ref{T:nocuspII}.
First note that since $C$ is immersed by \ref{T:nocuspI},
one has $N _\varphi = {\bar N} _\varphi$,
so that condition \ref{c:dim} implies the smoothness  
of the  scheme of morphisms $M_g(X)$ at a point $[\varphi]$ parametrizing $\varphi$,
the tangent space at this point being
\[
\H^0( N _\varphi )
= \H^0( \omega_{\bar C} \otimes \varphi^* \omega _X ^{-1} ).
\]

The assumption that $C$ is not nodal means that there is a point
$x \in C$ at which $C$ has 
\begin{inparaenum}
\renewcommand{\theenumi}{\normalfont (\roman{enumi})}
either \item
\label{case:triple} (at least) $3$ local branches meeting transversely,
or \item 
\label{case:tacnode} (at least) $2$ local branches tangent to each other.
\end{inparaenum}
In case \ref{case:triple}, there are three
pairwise distinct points $p,q,r \in \bar {C}$ such that
$\varphi(p)=\varphi(q)=\varphi(r)=x$.
It follows from condition~\ref{c:3ample} that there exists
a section $\sigma \in \H^0(N_{\varphi})$ such that 
$\sigma(p) =\sigma(q) =0$ and $\sigma(r) \neq 0$. 
Such a section corresponds to a first--order deformation
(hence, by smoothness, to an actual deformation) of $\varphi$
leaving both $\varphi(p)$ and $\varphi(q)$ fixed
while moving $\varphi(r)$: 
it therefore turns the triple point constituted at $x$ by the $3$
local branches of $C$ under consideration into $3$ nodes.
In particular it is not equisingular, a contradiction.

In case~\ref{case:tacnode},
there are $2$ distinct points $p,q \in \bar {C}$, such
that $\varphi(p)=\varphi(q)$, and 
$\im (d\varphi _p) = \im (d\varphi _q)$,
and it follows from condition~\ref{c:3ample} that there exists
a section $\sigma \in \H^0(N_{\varphi})$ such that 
$\sigma (p) =0$ and $\sigma (q) \not\in \im (d\varphi _p)$.
The corresponding deformation of $C$ leaves $\varphi (p)$ fixed
and moves $\varphi (q)$ in a direction transverse to the common
tangent to the $2$ local branches of $C$ under consideration
(if the $2$ branches of $C$ are simply tangent, the tacnode they
constitute at $x$ is turned into $2$ nodes).
It is therefore not equisingular either, a contradiction also
in this case.
\end{proof}

\paragraph{}
In many cases the conditions considered in Theorem~\ref{T:nocusp} are
not satisfied: this 
% can happen for instance when
% $\omega_X$ is trivial or ample. 
tends to happen when $\omega _X^{-1}$ is not positive enough.

\subparagraph{}
Clearly enough, \ref{c:bpf} does not hold in general.
Critical occurences of this phenomenon are to be observed for rational
curves on $K3$ surfaces (Remark~\ref{R:K3})
and for anticanonical rational curves on a degree $1$ Del Pezzo
surface (Remark~\ref{R:delpezzo}).
In these two situations, the conclusion of Theorem~\ref{T:nocusp} 
is not true in general.

\subparagraph{}
There can also be actual obstructions to deform the normalization of
the general member of $V$
and then \ref{c:dim} does not hold, see
Remark~\ref{Ex:nonredM} and Example~\ref{Ex:cayleybacharach}.
The conclusion of Theorem~\ref{T:nocusp} is not true 
for this example.

\begin{remark}
\label{R:toostrong}
Condition \ref{c:3ample}, albeit non-redundant
(see \ref{warn}), is too strong, as
the following example shows.
Let $(X,L)$ be a very general primitively polarized $K3$ surface, 
with $L^2 = 12$. It follows from Proposition~\ref{P:K3-gnlzddivs} that
the general element $C$ of every irreducible component of $V _{L,4}$ is 
nodal.
On the other hand, having genus $4$ the curve $\bar C$ is trigonal,
\ie there exists an effective divisor of degree $3$ on $\bar C$ 
such that $\h^0\bigl( \O _{\bar C}(A) \bigr) = 2$,
whence
\[
\h^0\bigl( \omega _{\bar C}(-A) \bigr) = 2
> \h^0\bigl( \omega _{\bar C} \bigr) -3 = 1,
\]
and condition~\ref{c:3ample} does not hold.

A finer analysis is required in order to get the right condition.
The approach described in section~\ref{S:cartesian} might provide a
possibility for doing so.
\end{remark}

The following result provides a convenient way to apply
Theorem~\ref{T:nocusp}.

\begin{corollary}\label{C:nocusp1}
Assume that $V\subset V^\xi_g$ is an irreducible component and let
$[C]\in V$ be  general. If
$\omega_{\bar{C}}\otimes\varphi^*\omega_X^{-1}$ is non-special and
base--point--free then $C$ has no cusps. If moreover
\begin{equation}\label{E:nocusps1}
\mathrm{deg}(\omega_{\bar{C}}\otimes\varphi^*\omega_X^{-1})\ge
2g+2
\end{equation}
then $C$ is nodal.
\end{corollary}

\begin{proof}
Condition \ref{c:bpf} of the theorem is satisfied by hypothesis. The
non-speciality of $\omega_{\bar{C}}\otimes\varphi^*\omega_X^{-1}$
implies that
\[
\chi(N_\varphi) \ge
h^0(\omega_{\bar{C}}\otimes\varphi^*\omega_X^{-1})
\]
and therefore also condition \ref{c:dim} is satisfied, thanks to
(\ref{E:mor3}). The last assertion is clear because
(\ref{E:nocusps1}) implies that condition \ref{c:3ample}
is also satisfied.
\end{proof}

\section{A Cartesian approach}
\label{S:cartesian}

The situation and notations are as set-up in 
subsection~\ref{s:setting}.

\subsection{Ideals defining tangent spaces}

\paragraph{}
Let $C$ be a reduced curve in the surface $X$.
We consider the sequence of sheaves of
ideals of $\O_C$
\[
  J \subseteq I \subseteq A \subseteq   \O_C,
\]
where:% \\
\begin{compactenum}
[\normalfont (\arabic{section}.\arabic{paragraph}.1)]
    \item $J$ is the jacobian ideal:
it is locally generated by the partial derivatives of a
local equation of $C$;%\\
    \item $I$ is the equisingular ideal \cite{jW74}: it does not have
      any non-deformation-theoretic interpretation;%\\
    \item $A$ is the adjoint ideal:
it is the conductor 
$\cond \nu := \shhom _{\O_C} (\nu_* \O_{\bar C}, \O_C)$
of the normalization
$\nu:\bar{C} \to C$ of $C$.
\end{compactenum}

\paragraph{}
Being $\nu$ birational, $\cond \nu$ is 
the annihilator ideal
$\shann _{\O_C} \bigl( \nu_* \O_{\bar C} / \O_C \bigr)$.
% of the sheaf of $\O_C$-modules $\nu_*\O_{\bar{C}}/\O_C$.
It follows that 
$A\subset \O_C$ is also a sheaf of ideals of
$\nu_*\O_{\bar{C}}$,
which implies that there exists an effective divisor $\bar{\Delta}$ on
$\bar{C}$ such that
\begin{equation}\label{E:adj1}
A \cong \nu_*\left(\O_{\bar{C}}(-\bar{\Delta})\right).
\end{equation}
Moreover, we have
\begin{equation}\label{E:adj2}
\omega_{\bar{C}} = \nu^*(\omega_C)\otimes\O_{\bar{C}}(-\bar{\Delta}).
\end{equation}

\begin{lemma}
\label{L:indep-adj}
For $i=0,1$, one has
 \[
\H^i\bigl(C,A\otimes\O_C(C)\bigr) \cong
\H^i\bigl(\bar{C}, \omega_{\bar{C}}\otimes
    \varphi^*\omega_X^{-1}\bigr),
    \]
where $\varphi:\bar C \to X$ is the composition of the normalization
map $\nu$ with the inclusion $C \subset X$.
\end{lemma}

\begin{proof}
By \eqref{E:adj1} and the projection formula, we have
\begin{align*}
\H^0\bigl(C,\O_C(C)\otimes A\bigr)
&= \H^0\bigl(C,\O_C(C)\otimes\nu_*\bigl(\O_{\bar{C}}(-\bar{\Delta})
  \bigr) \bigr)\\
&= \H^0\bigl(\bar{C},\nu^*\O_C(C)\otimes
\O_{\bar{C}}(-\bar{\Delta}) \bigr).
\end{align*}
 By \eqref{E:adj2} and the adjunction formula $\omega_C=\O_C(C)\otimes
 \omega_X$, we have
 \[
 \nu^*\O_C(C)\otimes \O_{\bar{C}}(-\bar{\Delta}) =
 \nu^*\omega_C\otimes \varphi^*\omega_X^{-1}\otimes 
\O_{\bar{C}}(-\bar{\Delta})=
 \omega_{\bar{C}}\otimes \varphi^*\omega_X^{-1},
 \]
 and  the  statement  follows  in  the
  case  $i=0$. 
For  the  second  identity,  observe 
that
 $R^1\nu_* ( \omega_{\bar{C}}\otimes \varphi^*\omega_X^{-1} )=0$,
hence
 \[
\H^1\bigl(C,\O_C(C)\otimes A\bigr)=
\H^1\bigl(C,\nu_* ( \omega_{\bar{C}}\otimes
\varphi^*\omega_X^{-1} ) \bigr) =
\H^1\bigl(\bar{C},  \omega_{\bar{C}}\otimes
\varphi^*\omega_X^{-1}\bigr)
 \]
by Leray's spectral sequence.
\end{proof}

\paragraph{}
% since $h^1(X,\O_X)=0$, there is a functorially defined
% identification between $T_{[C]}|L|$ and $\mathrm{H}^0(C,\O_C(C))$.
Let $\xi \in \NS(X)$ be the class of $C$.
From the functorial identification of
$\operatorname{\mathit T} _{[C]} \curves ^\xi_X$ with
$\H^0 \bigl( C, \O_C(C) \bigr)$ 
we may deduce the sequence of inclusions
\begin{equation}
\label{E:seq-tangents}
\H^0\bigl(C,J\otimes\O_C(C)\bigr) \subseteq
\H^0\bigl(C,I\otimes\O_C(C)\bigr) \subseteq
\H^0\bigl(C,A\otimes\O_C(C)\bigr) \subseteq
\operatorname{\mathit T} _{[C]} \curves ^\xi_X ,
%T_{C}|L|
\end{equation}
which has the following deformation-theoretic interpretation.

\begin{proposition}
[{\cite[Prop.~4.19]{DH88}}]

\subparagraph{}
\label{fact:loc-triv}
$\H^0(C,J\otimes\O_C(C))$ is the tangent space at $[C]$ to
    the subscheme of $\curves ^\xi_X$  of formally locally trivial
    deformations of $C$.

\subparagraph{}
\label{fact:equising}
$\H^0(C,I\otimes\O_C(C))$ is the tangent space at
    $[C]$ to $\ES (C)$.
In particular,
    \[
    \dim_{[C]}ES(C) \le h^0(C,I\otimes\O_C(C)).
    \]

\subparagraph{}
\label{fact:equigen}
$\H^0(C,A\otimes\O_C(C))$ contains the reduced tangent
    cone to $\V ^\xi _{g(C)}$  at $[C]$.
    In particular,
    \[
    \dim_{[C]}\V_{L,g} \le h^0(C,A\otimes\O_C(C))
= h^0\left( \bar C,
\omega_{\bar C} \otimes \varphi^* \omega_X^{-1} \right).
    \]
\end{proposition}

As is the case for $\ES(C)$, 
the subscheme of $\curves ^\xi_X$  of formally locally trivial
deformations of $C$ is functorially defined \cite{jW74b};
in contrast, $\V ^\xi _g$ is only set-theoretically defined.

\ref{fact:loc-triv} is based on results of Artin and Schlessinger
respectively; since we will not use this, we refer to \cite{DH88} for
the precise references.
\ref{fact:equising} follows from \cite{bT80,jW74}, as explained in
\cite[Prop.~3.3.1]{fL87}.
\ref{fact:equigen} stems from \cite{bT77}
(the last equality comes from Lemma \ref{L:indep-adj}).
\qed

\bigskip
The next result is 
% crucial in Zariski-Harris' approach to
% Question \ref{mainquest}.
conceptually important: it explains why one would envisage an
affirmative answer to Problem~\ref{Pb:main} in the first place.
\begin{proposition}
[{\cite{oZ82}}]
\label{P:notnodal}
Let $(C,p)$ be a reduced planar curve germ,
%Let $p \in C$, 
and consider the local ideals $I_p \subseteq A_p \subseteq
\O_{C,p}$. Then $I_p=A_p$  if and only if $p$ is a node.
\end{proposition}

This also occurs as \cite[Thm.~3.3.2]{fL87} and \cite[Lemma~6.3]{DH88},
where enlightening proofs are provided.

\subsection{Effective computations}

Next, we collect some results enabling one to compute in practice the
ideals $A$ and $I$ which will be needed in the sequel.

\begin{lemma}
[{\cite[II.6--7]{szpiro}}]
\label{L:szpiro}
Let $C \subset X$ be a reduced curve in a smooth surface.
Consider a finite chain of birational morphisms
\begin{equation*}
X _{s+1} \xrightarrow{\epsilon_s} X _{s}
\to \cdots \to
X_2 \xrightarrow{\epsilon_1} X_1 = X
\end{equation*}
such that each $\epsilon_r$ is the blow-up of a single closed point
$q_r \in X _{r}$, with exceptional divisor $E_r$
($1 \leq r \leq s$).
Let furthermore
\begin{compactitem}[-]
\setlength{\itemindent}{3mm}
\item
$\epsilon _{s,r} = \epsilon _r \circ \cdots \circ \epsilon _{s}: 
X _{s+1} \to X_r$,
\item
$C_r$ be the proper transform of $C$ in $X_r$ ($C_1=C$),
and \item
$m_r$ be the multiplicity of $C_r$ at $q_r \in X _{r}$ .
\end{compactitem}
If the proper transform %$(\epsilon _{s,0} ^{-1} )_* C$ 
of $C$ in $X _{s+1}$ is smooth, then the adjoint ideal of $C$ is
\begin{equation*}
A_C =
(\epsilon _{s,1})_* \O _{X_s} \Bigl(
-(m_1 -1) \epsilon _{s,2} ^* E_1 -\ \cdots\ 
-(m _{s-1} -1) \epsilon _{s,s} ^* E _{s-1}
- (m_s -1) E_s
\Bigr)
\otimes _{\O_X} \O_C .
\end{equation*}
\end{lemma}

\bigskip
As far as the equisingular ideal is concerned, we shall 
only need two special instances of \cite[Prop.~2.17]{GLSb},
and refer to loc.\ cit.\ \S\,II.2.2 
for further information.

\paragraph{}
Recall that a polynomial 
$f=\sum _{(\alpha,\beta) \in \N^2} a _{\alpha,\beta}\, x ^\alpha
y^\beta$
is said to be \emph{quasihomogeneous}
if there exist positive integers
$w_1,w_2,d$ such that
\[
\forall (\alpha,\beta) \in \N^2,\quad
a _{\alpha,\beta} \neq 0 \implies
w_1 \alpha + w_2 \beta =d.
\]
In such a case, $(w_1,w_2;d)$ is called the \emph{type} of $f$.
An isolated planar curve singularity $(C,0)$
is said to be \emph{quasihomogeneous}
if it is analytically equivalent to the singularity at the origin of a
plane affine curve defined by a quasihomogeneous polynomial $f$,
\ie if the complete local ring $\hat {\O} _{C,0}$ is isomorphic to
$\C [[x,y]] / \langle f \rangle$.

\begin{lemma}
[{\cite[Prop.~2.17]{GLSb}}] %p.287
\label{L:I}
Let $f \in \C[x,y]$ be a quasihomogeneous polynomial of type
$(w_1,w_2;d)$, and consider the affine plane curve $C$ defined
by $f$.
If $C$ has an isolated singularity at the origin $0$,
then the local equisingular ideal of $C$ at $0$ is
\begin{equation*}
I = J +
\langle x ^\alpha y ^\beta \ 
|\ 
w_1 \alpha + w_2 \beta \geq d \rangle.
\end{equation*}
in the local ring $\O _{C,0}$
(where $J$ denotes as usual the Jacobian ideal 
$\langle \partial _x f, \partial _y f \rangle$).
\end{lemma}

\paragraph{}
Recall that simple curve singularities are those defined by one of the 
following equations:
\begin{alignat*}{2}
A_\mu: &\hspace{5mm}& y^2 -x^{\mu+1}&=0 \quad (\mu\geq 1) \\
D_\mu: && x(y^2-x^{\mu-2})&=0 \quad (\mu\geq 4) \\
E_6: && y^3 - x^4 &=0 \\
E_7: && y(y^2-x^3)&=0 \\
E_8: && y^3 - x^5 &=0.
\end{alignat*}

Simple singularities are quasihomogeneous,
and one obtains as a corollary of Lemma~\ref{L:I} that 
\emph{
the equisingular ideal $I$ of a simple singularity equals its
jacobian ideal $J$.}
This means that simple singularities do not admit non 
topologically trivial equisingular deformations.

\begin{example}
\label{Ex:double}
\emph{(double points)}
Any double point of a curve is a simple singularity of type $A _\mu$,
$\mu \geq 1$.
At such a point $p$, the adjoint and equisingular ideals are 
respectively
\begin{equation*}
A = \langle y, x ^{\lfloor \frac{\mu+1}{2} \rfloor} \rangle
\quad \text{and} \quad
I = \langle y, x ^\mu \rangle
\end{equation*}
in the local ring of the curve at $p$.
\end{example}

\begin{example}
\label{ex:m-uple0}
\emph{(ordinary $m$-uple points)}
Let $m$ be a positive integer. An ordinary $m$-uple point of a curve
is analytically equivalent to the origin in an affine plane curve
defined by an equation
\begin{equation}
\label{eq:m-uple}
f(x,y) =
f_m (x,y) + \tilde{f} (x,y)
% x ^m + y ^m
% + a_0\, x ^{m-2} y^2
% + \cdots +
% a _{m-4}\, x^2 y ^{m-2}
= 0
\end{equation}
where $f_m$ is a degree $m$ homogeneous polynomial defining a smooth
subscheme of $\P^1$, and $\tilde{f}$ is a sum of monomials of degree
$>m$; such a polynomial $f$ is said to be 
\emph{semi-homogeneous}.
In particular, \cite[Prop.~2.17]{GLSb} applies to this situation,
and the adjoint and equisingular ideals at the origin of the curve
defined by \eqref{eq:m-uple} are respectively
\begin{equation*}
A = \langle x ^\alpha y ^\beta \
| \
\alpha+\beta \geq m-1 \rangle
\quad \text{and} \quad
I = \langle \partial _x f, \partial _y f \rangle
+ \langle x ^\alpha y ^\beta \
| \
\alpha+\beta \geq m \rangle
\end{equation*}
in the local ring of the curve
($A$ is computed with Lemma~\ref{L:szpiro}).
\end{example}

\subsection{Pull--back to the normalization}
\label{s:pull-back}

In this subsection, we discuss the possibility of proving 
Theorem~\ref{T:nocusp} by ``lifting'' the sequence of tangent spaces 
\eqref{E:seq-tangents} on the normalization of a general member of 
a maximal irreducible equigeneric family.
First of all, we would like to point out a fallacy:
we explain below why a certain line of argumentation does not enable
one to remove assumption~\ref{c:3ample} in \ref{T:nocuspII}. 
This incomplete argumentation is used in the proofs of
\cite[Prop.~2.1]{jH86} (last paragraph of p.448)
and of \cite[Lemma~3.1]{chen99} (last paragraph of the proof).
As indicated in \cite{jH86}, it is nevertheless possible to prove
\cite[Prop.~2.1]{jH86} using the parametric approach, see, e.g.,
\cite[p.105--117]{harris-morrison}. 
As for \cite[Lemma~3.1]{chen99} however, we do
not know of any valid proof.

\begin{warning}
\label{warn}
As in Theorem~\ref{T:nocusp}, consider an irreducible component
$V$ of $\V ^\xi _g$, and $[C]$ a general member of $V$, and assume
that conditions \ref{c:bpf} and \ref{c:dim} of \ref{T:nocusp}
hold. 
Suppose moreover that $C$ is not nodal;
this implies by \ref{P:notnodal} that $I_C \subsetneq A_C$.

Being $C\in V$ general, one has $T_{[C]}V \subset T_{[C]}\ES(C)$
by Proposition \ref{P:gen-equising}.
Therefore
\begin{multline}
\label{eq:warning}
\dim V 
\le 
\dim T_{[C]}V
\le 
\dim \bigl( \operatorname{\mathit{T}}_{[C]}\ES(C) \bigr) 
= %\le 
\h^0 \bigl(C,I_C \otimes \O_C(C) \bigr) \\
\le 
\h^0 \bigl(\bar C,\nu ^* \bigl(I_C \otimes \O_C(C)\bigr) \bigr)
= \h^0  \bigl(\bar{C},I'\otimes
\omega_{\bar{C}}\otimes\varphi^*\omega_X^{-1} \bigr), 
\end{multline}
where $I'$ is the ideal of $\O _{\bar C}$ determined by the relation
$\nu^*I_C = I'\otimes\nu^*A_C$
(as usual, $\nu: \bar{C} \to C$ is the normalization of $C$
and $\varphi$ its composition with the inclusion 
$C \subset X$).

Now:
\emph
{although $\omega_{\bar{C}}\otimes\varphi^*\omega_X^{-1}$ is globally
  generated by our hypothesis \ref{c:bpf}
and $I_C \subsetneq A_C$ because $C$ is not nodal,
in general it does not follow from the sequence of inequalities
\eqref{eq:warning} that
$\dim V < 
h^0(\bar{C}, \omega_{\bar{C}}\otimes \varphi^*\omega_X^{-1})$,
\ie there is a priori no contradiction with assumption \ref{c:dim}.
}

The reason for this is that $\nu^* I_C$ and $\nu^* A_C$
may be equal even if $I_C$ and $A_C$ are not
(see Examples~\ref{ex:segre} and \ref{ex:m-uple} below).
In such a situation, $I'$ is trivial, and \eqref{eq:warning}
only gives $\dim V \leq 
h^0(\bar{C}, \omega_{\bar{C}}\otimes \varphi^*\omega_X^{-1})$.
Example~\ref{Ex:jacobian} displays a situation when both \ref{c:bpf}
and \ref{c:dim} hold, but the general member $C \in V$ is not nodal
(\ie conditions \ref{c:bpf} and \ref{c:dim} hold but the
conclusion of \ref{T:nocuspII} doesn't):
in this example one has
$\H^0(I_C \otimes \O_C (C)) = \H^0(A_C \otimes \O_C (C))$
although $\nu^* (A_C \otimes \O_C (C))$ is globally generated
and $I_C \subsetneq A_C$.
Therefore, condition~\ref{c:3ample} of \ref{T:nocuspII}
is not redundant.

With this respect, it is important to keep in mind 
that base--point--freeness of the linear system
$\nu^* \bigl| A \otimes \O_C (C) \bigr|$
on $\bar C$
does not imply base--point--freeness of the linear system
(of generalized divisors, see \ref{p:gnlzd-linsys}
below)
$\bigl| A \otimes \O_C (C) \bigr|$
on $C$.
And indeed, it is almost always the case that 
$\bigl| A \otimes \O_C (C) \bigr|$ has base points
(see Remark~\ref{r:basepoints}).

Also, note that 
the linear subsystem
$\nu ^* \bigl| I_C \otimes \O_C(C)\bigr|$ of
$\bigl|\nu ^* \bigl(I_C \otimes \O_C(C)\bigr) \bigr|$
is in general not complete
(see Example~\ref{ex:segre}),
in contrast with the fact that
$\nu^* \bigl| A \otimes \O_C (C) \bigr|
= \bigl| \nu^* \bigl( A \otimes \O_C (C) \bigr) \bigr|$
by independence of the adjoint conditions (Lemma~\ref{L:indep-adj}).
\end{warning}

\begin{example}\label{ex:segre}
\cite{S-tac}
Let $C \subset \P^2$ be a degree $n$ curve with one ordinary tacnode 
(\ie a singularity of type $A_3$) at a point $p$
and smooth otherwise.
At $p$, there are local holomorphic coordinates $(x,y)$ such that $C$ 
has equation $y^2 = x^4$. Then
\[
A _{C,p} = \langle y,x^2 \rangle
\quad \text{and} \quad
I _{C,p} = \langle y,x^3 \rangle,
\]
(see Example~\ref{Ex:double})
whence the linear system
$| A_C \otimes \O_C(C)|$
(resp. $| I_C \otimes \O_C(C)|$) on $C$
is cut out by the system of degree $n$ curves 
tangent at $p$ to the two local branches of $C$ there
(resp. having third order contact at $p$ with the reduced tangent
cone to $C$ there).

Now, a third order contact with the reduced tangent cone at $p$ does 
not imply anything beyond simple tangency with each of the two local
branches of $C$ there.
In coordinates, this translates into the fact that
\[
\nu^* A _{C,p_i} = 
\nu^* I _{C,p_i} = \langle t_i^2 \rangle
\]
at the two preimages $p_i$, $i=1,2$, of $p$, $t_i$ being a local
holomorphic coordinate of $\bar C$ at $p_i$.
Nevertheless the linear system $\nu^* | I_C \otimes \O_C(C)|$ has
codimension $1$ in $|\nu ^* \bigl(I_C \otimes \O_C(C)\bigr) |
= |\nu ^* \bigl(A_C \otimes \O_C(C)\bigr) |$
and is free from base point.
\end{example}

\begin{example}
\label{ex:m-uple}
We consider an ordinary $m$-uple planar curve singularity $(C,0)$ as
in Example~\ref{ex:m-uple0}.
Without loss of generality, we assume that the line $x=0$ is not
contained in the tangent cone of $C$ at $0$.
Then $x$ is a local parameter for each local branch, and
it follows from the computations of $A _{C,0}$ and $I _{C,0}$ in
Example~\ref{ex:m-uple0} that 
\[
\nu ^* A _{C,0} = \nu ^* I _{C,0}
= \langle x ^{m-1} \rangle,
\]
where $\nu$ is the normalization of $C$.
\end{example}

\paragraph{}
It might nevertheless be possible to use the argument given in 
\ref{warn} to give another proof of \ref{T:nocuspI},
\ie  that \ref{c:bpf} and \ref{c:dim} of 
Theorem~\ref{T:nocusp} imply that the general member of $V$ is
immersed.
We have indeed not found any example of a non immersed planar curve
singularity such that the pull--backs by the normalization of $I$ and
$A$ are equal. 
The next statement is a first step in this direction.

\begin{remark}
Let $(C,0)$ be a simple curve singularity, and $\nu$ its
normalization.
Then 
$\nu ^* A _{C,0} \neq \nu ^* I _{C,0}$
if and only if $(C,0)$ is not immersed.
\end{remark}

\begin{proof}
This is a basic computation. We treat the case of $E_8$, and leave the
remaining ones to the reader.%
\footnote{actually, they are hidden as comments in the .tex file}
The normalization $\nu$ of the $E_8$ singularity factors as
the sequence of blow--ups
\[\xymatrix@R=1mm{
\displaystyle{
\frac {\C[u_1,v_2]} {\langle u_1 - v_2 ^2 \rangle}
}
& \displaystyle{
\frac {\C[x,u_1]} {\langle x^2 - u_1 ^3 \rangle}
}
\ar[l] _{\epsilon_2}
& \displaystyle{
\frac {\C[x,y]} {\langle y^3 -x^5 \rangle}
}
\ar[l] _{\epsilon_1}
\\
(v_2\, u_1, u_1)
& (x,u_1)\ ; \ar@{|->}[l]
(x,u_1\,x)
& (x,y)
\ar@{|->}[l]
}\]
and it follows from Lemma~\ref{L:szpiro} that its adjoint is
\begin{align*}
A= (\epsilon_1) _* \langle x,u_1 \rangle \cdot
\langle x,y \rangle ^2
&= \langle x^3, x^2y, xy^2,
u_1x^2, u_1xy, u_1y^2 \rangle
\\
&= \langle x^3, x^2y, xy^2,
yx, y^2, x^4 \rangle
% & (u_1 y^2 = u_1 (u_1 x)^2  = u_1^3 x^2 = x^2 x^2)
\\
&= \langle x^3, xy, y^2 \rangle.
\end{align*}
% \[
% A= (\epsilon_1) _* \langle x,u_1 \rangle \cdot
% \langle x,y \rangle ^2
% = \langle x^3, xy, y^2 \rangle.
% \]
On the other hand, its equisingular ideal is
$
I = J =
\langle x^4,y^2 \rangle
$
by Lemma~\ref{L:I}.
Eventually, one has
\begin{equation*}
\begin{aligned}
\nu^* A &= \epsilon_2 ^* \epsilon_1 ^* \langle x^3, xy, y^2 \rangle \\
&= \epsilon_2 ^* \langle x^3, u_1 x^2, u_1^2 x^2 \rangle
= \epsilon_2 ^* \langle x^3, u_1 x^2 \rangle \\
&= \langle v_2^3 u_1^3, v_2^2 u_1^3 \rangle
= \langle v_2^2 u_1^3 \rangle = \langle v_2 ^8 \rangle
\end{aligned}
\qquad \text{and} \qquad
\begin{aligned}
\nu^* I &= \epsilon_2 ^* \epsilon_1 ^* \langle x^4, y^2 \rangle \\
&= \epsilon_2 ^* \langle x^4, u_1^2 x^2 \rangle \\
&= \langle v_2 ^4 u_1 ^4, v_2 ^2 u_1 ^4 \rangle
= \langle v_2 ^2 u_1 ^4 \rangle = \langle v_2 ^{10} \rangle,
\end{aligned}
\end{equation*}
so that $\nu^* A \neq \nu^* I$, and indeed the $E_8$ singularity is
non--immersed. 
\end{proof}

\begin{complimentary-computation}
\textbullet\
We now consider an $A _{2n+r-1}$ singularity, $r=0$ or $1$ and $n\geq 1$
(i.e., the equation is $y^2=x ^{2n+r}$).
This is resolved with the series of $n$ blow--ups
\[\xymatrix@R=1mm@C=7mm{
\displaystyle{
\frac {\C[x,u_{n}]} {\langle u_{n} ^2 - x ^{r} \rangle}
}
& \displaystyle{
\frac {\C[x,u_{n-1}]} {\langle u_{n-1} ^2 - x ^{2+r} \rangle}
}
\ar[l] _ (.52) {\epsilon_n}
& \cdots \ar[l]
& \displaystyle{
\frac {\C[x,u_1]} {\langle u_1 ^2 - x^{2(n-1)+r} \rangle}
}
\ar[l] _ (.65) {\epsilon_2}
& \displaystyle{
\frac {\C[x,y]} {\langle y^2 -x^{2n+r} \rangle}
}
\ar[l] _(.42) {\epsilon_1}
\\
(x, u_n\, x)
& (x, u_{n-1}) \ar@{|->}[l]
& (x, u_2\, x) \ar @{} [l] |{\displaystyle{\cdots}}
& (x,u_1)\ ; \ar@{|->}[l]
(x,u_1\,x)
& (x,y)
\ar@{|->}[l]
}\]
and consequently 
\begin{align*}
A &= (\epsilon _{n-1,1}) _* \langle x, u_{n-1} \rangle \cdot
(\epsilon _{n-2,1}) _* \langle x, u_{n-2} \rangle \cdots
%(\epsilon_1)_* \langle x, u_{1} \rangle \cdot
\langle x, y \rangle 
\\
&= (\epsilon _{n-2,1}) _* \langle x^2, u_{n-2} \rangle \cdots
\langle x, y \rangle 
\\
&\hspace{2mm} \vdots \\
&= \langle x^n, y \rangle .
\end{align*}
On the other hand $I=J= \langle x ^{2n+r-1}, y \rangle$,
and eventually
\begin{equation*}
\begin{aligned}
\nu^* A &= \epsilon_n ^* \cdots \epsilon_1 ^* \langle x^n, y \rangle \\
&= \epsilon_n ^* \cdots \epsilon_2 ^* \langle x^n, u_1 x \rangle \\
&\hspace{2mm} \vdots \\
& %= \epsilon_n ^* \langle x^n, u_{n-1} x^{n-1} \rangle 
= \langle x^n, u_n x^n \rangle
\end{aligned}
\quad \qquad \text{and} \quad \qquad
\begin{aligned}
\nu^* I &= \epsilon_n ^* \cdots \epsilon_1 ^* 
\langle x^{2n+r-1}, y \rangle \\
&= \epsilon_n ^* \cdots \epsilon_2 ^* \langle x^{2n+r-1}, u_1 x \rangle \\
&\hspace{2mm} \vdots \\
& %= \epsilon_n ^* \langle x^{2n+r-1}, u_{n-1} x^{n-1} \rangle 
= \langle x^{2n+r-1}, u_n x^n \rangle.
\end{aligned}
\end{equation*}
Thus
\[
\nu^* I = 
\begin{cases}
\langle x^n \rangle & \text{if}\ r=0 \\
\langle u_n ^{2n+1} \rangle & \text{if}\ r=1
\end{cases}
\qquad \text{whereas} \qquad
\nu^* A
 = 
\begin{cases}
\langle x^n \rangle & \text{if}\ r=0 \\
\langle u_n ^{2n} \rangle & \text{if}\ r=1,
\end{cases}
\]
whence $\nu^* A \neq \nu^* I$ if and only if $r=1$,
\ie if and only if the singularity is non--immersed.

\bigskip
\textbullet\
Similarly enough, the $D _{2n+2+r}$ singularity
($n \geq 1$, $r=0$ or $1$) is resolved in $1+(n-1)$ blow--ups.
Write the first one
\[
\Proj
\frac {\C[x,y][u,v]}
{\langle v\,(u^2 - v^2\,x^{2(n-1)+r}), ux-vy \rangle}
\xrightarrow{\epsilon_0}
\operatorname{Spec}
\frac {\C[x,y]}
{\langle x(y^2 - x^{2n+r}) \rangle},
\]
the $\operatorname{Proj}$ construction being done with respect to the
grading defined by letting $\deg x = \deg y =0$ and 
$\deg u = \deg v =1$.
It follows from Lemma~\ref{L:szpiro} and the computations already
performed for the $A _\mu$ singularities that
\begin{align*}
A = (\epsilon_0)_* \langle u, x^{n-1} \rangle
\langle x,y \rangle ^2
&= \langle
ux^2, uxy, uy^2,
x ^{n+1}, x^n y, x^{n-1} y^2
\rangle \\
&= \langle
xy, y^2, x ^{2n-1+r} y,
x ^{n+1}, x^n y, x^{n-1} y^2
\rangle \\
& = \langle
x ^{n+1}, xy, y^2
\rangle,
\end{align*}
while
\[
I = J =
\langle y^2 - (2n+1+r) x^{2n+r}, xy \rangle.
\]
The pull--backs of $A$ and $I$ on the (smooth) branch $v=0$ are
always equal, namely one has
\[
\left. \epsilon_0 ^* A \right| _{v=0}
= \left. \epsilon_0 ^* I \right| _{v=0}
= \langle y^2 \rangle.
\]
On the other hand, writing the $n-1$ remaining blow--ups
\[\xymatrix@R=1mm@C=7mm{
\displaystyle{
\frac {\C[x,s_{n-1}]} {\langle s_{n-1} ^2 - x ^{r} \rangle}
}
& \cdots \ar[l] _ (.35) {\epsilon _{n-1}}
& \displaystyle{
\frac {\C[x,s_1]} {\langle s_1 ^2 - x^{2(n-2)+r} \rangle}
}
\ar[l] _ (.65) {\epsilon_2}
& \displaystyle{
\frac {\C[x,u]} {\langle u^2 -x^{2(n-1)+r} \rangle}
}
\ar[l] _(.47) {\epsilon_1}
\\
(x, s_{n-1}\, x) 
& (x, s_{n-2}\, x) \ar@{|->}[l]
& %(x,s_1)\ ; 
%\ar@{|->}[l]
(x,s_1\,x)
\ar @{} [l] |{\displaystyle{\cdots}}
& (x,u)
\ar@{|->}[l]
}\]
one computes the pull--back(s) on the other branch(es) 
\[
\begin{aligned}
\nu^* A &= \epsilon_{n-1} ^* \cdots \epsilon_1 ^* \epsilon_0 ^*
\langle x^{n+1}, xy, y^2 \rangle \\
&= \epsilon_{n-1} ^* \cdots \epsilon_1 ^* 
\langle x^{n+1}, ux^2 \rangle \\
&= \epsilon_{n-1} ^* \cdots \epsilon_2
\langle x^{n+1}, s_1x^3 \rangle \\
&\hspace{2mm} \vdots \\
& = \langle x ^{n+1}, s _{n-1} x^{n+1} \rangle
\end{aligned}
\quad \qquad \text{and} \qquad
\begin{aligned}
\nu^* I &= \epsilon_n ^* \cdots \epsilon_1 ^*  \epsilon_0 ^*
\langle y^2 - (2n+1+r)x^{2n+r}, xy \rangle \\
&= \epsilon_n ^* \cdots \epsilon_1 ^*
\langle x^{2n+r}, ux^2 \rangle \\
&= \epsilon_{n-1} ^* \cdots \epsilon_2
\langle x^{2n+r}, s_1x^3 \rangle \\
&\hspace{2mm} \vdots \\
& = \langle x ^{2n+r}, s _{n-1} x^{n+1} \rangle.
\end{aligned}
\]
When $r=0$, $s_{n-1}$ is invertible, and one has 
$\nu^* A = \nu^* I = \langle x^{n+1} \rangle$ on the two branches
under consideration.
When $r=1$, $s _{n-1}$ is a local parameter over the singularity for
the (unique) 
branch of the normalization under consideration, and one has
$\nu ^* I = \langle s _{n-1} ^{2n+3} \rangle
\subsetneq \nu^* A= \langle s _{n-1} ^{2n+2} \rangle$
on this branch.

\bigskip
\textbullet\
The $E_6$ singularity is resolved with the single blow--up
\[\xymatrix@R=1mm@C=7mm{
\displaystyle{
\frac {\C[x,u]} {\langle u^3-x \rangle}
}
& \displaystyle{
\frac {\C[x,y]} {\langle y^3-x^4 \rangle}
}
\ar[l]_{\epsilon}
\\
(x,ux)
& (x,y)
\ar @{|->} [l]
}\]
and one has
$A=\langle x,y \rangle ^2$
and 
$I=J= \langle x^3,y^2 \rangle$.
Consequently,
\[
\nu^* I = \langle u^8 \rangle
\subsetneq
\nu^* A = \langle u^6 \rangle,
\]
and indeed $E_6$ is a cuspidal singularity.

\bigskip
\textbullet\
The $E_7$ singularity is resolved by the series of two blow--ups
\[\xymatrix@R=1mm@C=7mm{
\displaystyle{
\Proj \frac {\C[x,s][u,v]}
{\langle u\,(us-v), ux-vs \rangle}
}
& \displaystyle{
\Spec \frac {\C[x,s]} {\langle s(s^2-x) \rangle}
}
\ar @{<-} [l] _ (.4) {\epsilon_2}
& \displaystyle{
\Spec \frac {\C[x,y]} {\langle y(y^2-x^3) \rangle}
}
\ar @{<-} [l]_{\epsilon_1}
\\
& (x,sx)
& (x,y)
\ar @{|->} [l]
}\]
where again the $\Proj$ construction is done after assigning 
degree $1$ (resp. $0$) to the variables $u$ and $v$
(resp. $x$ and $s$).
One has
\begin{align*}
A
= (\epsilon_1)_* \langle s,x \rangle \cdot
\langle x,y \rangle ^2 
&= \langle sx^2,sxy,sy^2, x^3,x^2y,xy^2
\rangle \\
&= \langle xy,y^2,yx^2, x^3,x^2y,xy^2
\rangle \\
&= \langle x^3, xy, y^2 \rangle
\end{align*}
and $I=\langle x^2y, 3y^2-x^3 \rangle$,
whence
$\epsilon_1 ^* A = \langle x^3,sx^2 \rangle$
and
$\epsilon_1 ^* I = \langle s x^3, 3s^2 x^2 - x^3 \rangle$.
It follows that
\[
\left. \nu^* A \right| _{u=0}
= \left. \nu^* I \right| _{u=0}
= \langle x^3 \rangle
\quad \text{but} \quad
\left. \nu^* I \right| _{us=v}
= \langle s^6 \rangle
\subsetneq
\left. \nu^* A \right| _{us=v}
= \langle s^5 \rangle
\]
(note that the branch $us=v$ canonically identifies with
$\Spec \C[x,s] / \langle x-s^2 \rangle$),
and indeed $E_7$ has a cuspidal branch.
\end{complimentary-computation}

\begin{remark}
In any event, the tendency is that one loses information during the
pull--back, even in the case of non--immersed singularities.
For instance, in the case of an $A _{2n}$ singularity one has
$\dim _{\C} \nu^* A / \nu^* I = 1$ whereas
$\dim _{\C} A / I = n$.
\end{remark}

\subsection{Generalized divisors}
\label{s:gnlzddivs}

As explained in the previous subsection, one loses information as one
pulls back the strict inclusion $I \subsetneq A$ to the
normalization. 
In other words, in order to exploit the full strength of this 
inequality,
it is required to work directly on the singular
curve under consideration.
Here, we describe 
a possibility for doing so, namely by using the theory of generalized
divisors on curves with Gorenstein singularities
(a condition obviously fulfilled by divisors on smooth surfaces),
as developed by Hartshorne \cite{Hgnlzd-div}.
A meaningful application will be given for $K3$ surfaces 
in subsection~\ref{s:app-triv}.

\paragraph{}
Recall from \cite[\S1]{Hgnlzd-div} that generalized divisors on an
integral curve $C$ with Gorenstein singularities are defined as being
\emph{fractional ideals} of $C$,
\ie as those nonzero subsheaves of $\mathcal{K}_C$
(the constant sheaf of the function field of $C$)
that are coherent $\O_C$--modules;
note that fractional ideals of $C$ are rank $1$ torsion--free coherent 
$\O_C$--modules.
As a particular case, 
nonzero coherent sheaves of ideals of $\O_C$ are generalized divisors;
these correspond to $0$--dimensional subschemes of $C$, and
are called \emph{effective} generalized divisors.

The addition of a generalized divisor and a Cartier divisor is 
well-defined (and is a generalized divisor), but there is no
reasonable way to define the addition of two generalized divisors.
There is an inverse mapping $D \mapsto -D$, which at the level of
fractional ideals reads
$\mathcal{I} \mapsto \mathcal{I} ^{-1}
:= \lbrace f\in \mathcal{K}_C\ |\ 
f\cdot \mathcal{I} \subset \O_C \rbrace$.
Hartshorne moreover defines a \emph{degree} function on the set of
generalized divisors, which in the case of a
$0$--dimensional subscheme $Z$ coincides with the length of $\O _Z$.
He then shows that both the Riemann--Roch formula and Serre duality
hold in this context.

\paragraph{}
\label{p:gnlzd-linsys}
Let $Z$ be a generalized divisor on $C$, and $\O_C (Z)$ the inverse of
the fractional ideal corresponding to $Z$. The projective space of
lines in $\H^0 (\O_C (Z))$ is in bijection with the set $|Z|$
of effective generalized divisors linearly equivalent to $Z$.
A point $p \in C$ is a base point of $|Z|$ if
$p \in \operatorname{Supp} Z'$ for every $Z' \in |Z|$. 
One has to be careful that $\O_C (Z)$ may be generated by global
sections even though $|Z|$ has base points, 
and that it is in general not possible to associate to $|Z|$ a
base--point--free linear system by subtracting its base locus, the
latter 
being a generalized divisor,
see \cite[p.378--379 and Ex.~(1.6.1)]{Hgnlzd-div}.

\paragraph{}
Let $C$ be an integral curve in a smooth surface $X$, 
and $\xi$ its class in $\NS(X)$.
The adjoint and equisingular ideals $A$ and $I$ of $C$ 
define two effective generalized divisors on $C$, which we shall denote
respectively by $\Delta$ and $E$.
As a reformulation of Proposition~\ref{P:notnodal}, we have:
\begin{equation*}
%\label{gnlzdcond-notnodal}
C\ \text{not nodal}
\quad \iff \quad
\deg E > \deg \Delta.
\end{equation*}
Now, to argue along the lines of \ref{warn}, one has to estimate
\begin{equation*}
% \dim \bigl( T_{[C]} \V ^\xi _{g(C)} \bigr)
% - \dim \bigl( \operatorname{\mathit{T}}_{[C]}\ES(C) \bigr) 
\h^0 \bigl(C, N _{C/X} \otimes A \bigr)
- \h^0 \bigl(C, N _{C/X} \otimes I \bigr)
= \h^0 \bigl(C, \O_C (C-\Delta) \bigr)
- \h^0 \bigl(C, \O_C (C-E) \bigr).
\end{equation*}

\begin{lemma}
\label{L:notnodal2}
If in the above situation $C$ is not nodal, then
\begin{equation*}
\h^0 \bigl(C, N _{C/X} \otimes A \bigr)
- \h^0 \bigl(C, N _{C/X} \otimes I \bigr)
> \h^0 \bigl(C, \omega_X \otimes \O_C (\Delta) \bigr)
- \h^0 \bigl(C, \omega_X \otimes \O_C (E) \bigr).
\end{equation*}
\end{lemma}

\begin{proof}
By the Riemann--Roch formula together with Serre duality and the
adjunction formula, 
%$\omega_C = \omega_X \otimes \O_C (C)$,
we have
\begin{multline*}
\big[
\h^0 \bigl(C, N _{C/X} \otimes A \bigr)
- \h^0 \bigl(C, \omega_X \otimes \O_C (\Delta) \bigr)
\big]
- \big[
\h^0 \bigl(C, N _{C/X} \otimes I \bigr)
- \h^0 \bigl(C, \omega_X \otimes \O_C (E) \bigr)
\big]
\\ 
= -\deg \Delta + \deg E.
\end{multline*}
\end{proof}

\paragraph{}
\label{p:rigid}
As a sideremark (which will nevertheless be useful in our application
to $K3$ surfaces), note that
$\deg \Delta = p_a(C)-g(C)$, the so--called $\delta$--invariant of the
curve $C$.
Moreover, it follows from Serre duality and Lemma~\ref{L:indep-adj}
that
\[\h^1(C,\O_C (\Delta))=
\h^0(C,\omega_C (-\Delta))=
%\h^0(\bar{C}, \omega _{\bar{C}}) = 
g(C).\]
The Riemann--Roch formula then tells us that 
$\h^0(\O_C(\Delta))=1$,
\ie $\Delta$ is a rigid (generalized) divisor.

\begin{remark}
\label{r:basepoints}
The linear system $|N_{C/X} \otimes A|$ has almost always base points.
To see why, consider the typical case when $C$ has an ordinary 
$m$-uple point $p$ and no further singularity.
Then it follows from Example~\ref{ex:m-uple0} that 
$|N_{C/X} \otimes A|$ consists of those effective generalized divisors
linearly equivalent to $N_{C/X} -(m-1)p$.
Now, every effective divisor linearly equivalent to $N_{C/X}$ and
containing $p$ has to contain it with multiplicity $\geq m$,
so that $|N_{C/X} \otimes A|$ has $p$ as a base point.
\end{remark}

\section{Applications}
\label{S:apps}

\paragraph{}
Historically, the first instance of Problem~\ref{Pb:main}
to be studied was that of curves in the projective plane, by Zariski
on the one hand, and by Arbarello and Cornalba on the other.
In this situation, the parametric approach of 
Section~\ref{S:parametric} can be efficiently applied.

Usually, inequality~\ref{c:dim} of Theorem~\ref{T:nocusp}
is obtained from the estimate
    \begin{equation}
\label{E:g^2_d}
    \dim \bigl(\mathcal{G}^2_d \bigr) \ge 3d+g-9
    \end{equation}
proved in \cite{AC80}, where $\mathcal{G}^2_d$ is the moduli space
of pairs $(C,V)$ consisting of a genus $g$ (smooth projective) 
curve $C$ and of a $g^2_d$ on $C$ (i.e., $V$ is a degree $d$ linear
system of dimension $2$ on $C$),
together with the fact that the group of projective
transformations of the plane has dimension $8$. 
As a sideremark, note that 
\[
3d+g-9 = 
%(3g-3)+(3d-2g-6)
\dim {\cal M}_g
+\rho(2,d,g),
\]
where ${\cal M}_g$ is the moduli space of genus $g$ curves,
and $\rho(r,d,g)= g-(r+1)(g+r-d)$ is the Brill-Noether number
(see \cite[p.159]{ACGH}).

In subsection~\ref{s:app-rat} below, 
we deduce inequality~\ref{c:dim} of \ref{T:nocusp}
in a more abstract nonsensical way from (\ref{E:mor3}), 
which actually shows that we have equality in \ref{c:dim}
of \ref{T:nocusp}, hence
also in \eqref{E:g^2_d}, even when $\rho<0$.

\subsection{Applications to rational surfaces}
\label{s:app-rat}

We now collect various applications of 
Corollary~\ref{C:nocusp1} that settle Problem~\ref{Pb:main} for common
rational surfaces.
The paper \cite{KS13} contains results going the same direction.

\paragraph{}
We make repeated use of the
%One essential, albeit elementary, point 
elementary fact that any line bundle of
degree $\geq 2g$ on a smooth genus $g$ curve is non--special and
globally generated.

\begin{corollary}
[{\cite{AC80,oZ82}}]
\label{C:ACZ0}
The general element of the Severi variety  $\V_{d,g}$ of 
integral plane curves of degree $d$ and genus $g$ is 
a nodal curve.
\end{corollary}

\begin{proof}
This is trivial for $d=1$, and if $d \geq 2$, one has
for $[C]\in \V_{d,g}$
%$\varphi^*\omega_{\P^2}^{-1}=\varphi^*\O_{\P^2}(3)$, hence
\[
\deg(\omega_{\bar{C}}\otimes\varphi^*\omega_X^{-1}) 
= 2g-2+3d
\ge 2g+2,
\]
whence Corollary~\ref{C:nocusp1} applies.
\end{proof}

\begin{corollary}\label{C:ACZ-DP}
Let $X$ be a Del Pezzo surface of degree $d$, i.e. $-K_X$ is ample and $K_X^2=d$.
Then for every $n\geq 1$, the general element $C$ of any irreducible
component of  $\V_{-n K_X,g}$ is nodal,
unless $dn \leq 3$.
In any event, $C$ is immersed 
unless $d=n=1$ and $g=0$.
\end{corollary}

\begin{proof}
For $[C] \in \V_{-n K_X,g}$ we have
\[
\deg(\omega_{\bar{C}}\otimes\varphi^*\omega_X^{-1})= 2g-2+nd,
\]
which is $\geq 2g+2$ if $nd \geq 4$ 
and $\geq 2g$ if $nd \geq 2$,
so that Corollary~\ref{C:nocusp1} applies.
When $d=n=1$, we are considering the pencil $|-K_X|$,
the general member of which
is a smooth irreducible curve of genus 1.
\end{proof}

\begin{remark}
\label{R:delpezzo}
Observe that the case of $\V _{-K_X,0}$ when $X$ is a Del Pezzo surface
of degree $1$ is a true exception, as the following example
shows.

Let $D \subset \P^2$ be an irreducible cuspidal cubic, and let $X$
be the blow-up of  $\P^2$ at eight of the nine points of
intersection of $D$ with a general cubic. The proper transform $C$
of $D$ is isolated in $\V_{-K_X,0}$, and is not nodal.
In fact  $\bar{C}=\P^1$, and
\[
\h^0\left(\bar{C},\omega_{\bar{C}}\otimes\varphi^*\omega_X^{-1}\right)=
\h^0\left(\P^1,\O_{\P^1}(-1)\right)=0=
\dim\bigl(\V_{\omega_X^{-1},0}\bigr),
\]
but  $\omega_{\bar{C}}\otimes\varphi^*\omega_X^{-1}=\O_{\P^1}(-1)$ is
not globally generated,
so that Theorem~\ref{T:nocusp} does not apply.

It is remarkable that, when unlike the above situation
$X$ is $\P^2$ blown-up
at eight general enough points, all
members of $\V_{-K_X,0}$ are nodal curves.
\end{remark}

\begin{corollary}
\label{C:ACZ-Hirzebruch}
Let $X:={\bf F}_n=\P\left(\O_{\P^1}\oplus\O_{\P^1}(n)\right)$
be a Hirzebruch surface ($n \geq 0$).
For every effective $L \in \Pic X$ and 
$0 \le g \le p_a(L)$, 
the general member of every irreducible component of
$V_{L,g}$ is a nodal curve.
\end{corollary}

\begin{proof}
Let $E$ and $F$ be the respective linear equivalence classes of 
the exceptional section and a fibre of the ruling, and
$H=E+nF$.
It is enough to consider the case $L=dH+kF$, $d,k \geq 0$, since every
effective divisor on $X$ not containing the exceptional section
belongs to such an $|L|$.
Consider an integral curve $C \in |L|$ of genus $g$.
One has
\[
\deg \left( \omega_{\bar C} \otimes \varphi^* \omega_X^{-1} \right) =
2g-2-K_X \cdot C 
= 2g-2+dn+2d+2k,
\]
which is $\geq 2g+2$ 
(so that Corollary~\ref{C:nocusp1} safely applies),
unless
either $d=0$ and $k=1$
or $d=1$, $k=0$ and $n \leq 1$.
An elementary case by case analysis shows that the latter cases are
all trivial.
\end{proof}

\subsection{Applications to surfaces with numerically trivial
canonical bundle}
\label{s:app-triv}

We now deal with the case when $K_X \equiv 0$.
In this situation Corollary~\ref{C:nocusp1} does not apply directly
and further arguments are required.

\subsubsection{$K3$ surfaces}

Let $(X,L)$ be a polarized $K3$ surface, with $L^2=2p-2$,
$p\ge 2$, and let $0\le g \le p$.
Then $X$ is regular, and $p$ equals both
the dimension of $|L|$ and the arithmetic genus of a member of
this linear system.
Moreover, it follows from \ref{L:comp1} and \eqref{E:mor3} that
\begin{equation}
\label{ineq:estim0}
g-1 \leqslant \dim \V _{L,g} \leqslant g.
\end{equation}
In this case, the existence of deformations of projective $K3$
surfaces into non algebraic ones enables one to refine the
former dimension estimate, still by using the techniques of
\S\ref{S:morph}. This is well-known to the experts. We shall 
nevertheless prove
it here for the sake of completeness, along the lines of
\cite[Exercise II.1.13.1]{jK96} and \cite[Corollary 4]{zR93}.

\begin{proposition}\label{P:estim}
% Let $(X,L)$ be a polarized $K3$ surface, with $L^2=2p_a-2$,
% $p_a\ge 2$, and let $0\le g \le p_a$.  Then 
Every irreducible component
$V$ of $\V_{L,g}$ has dimension $g$.
\end{proposition}

\begin{proof}
Using Lemma \ref{L:comp1} and inequality \eqref{ineq:estim0}, it suffices
to prove that for every irreducible component $M$ of $M_g(X)$ 
and general $[\phi: D \to X] \in M$, one has
\begin{equation}
\label{ineq:required}
\dim_{[\phi]} M \geq g+\dim(\Aut D).
\end{equation}
We consider $\X \to \Delta$ an analytic deformation of $X$
parametrized by the disc,  
such that the fibre over $t \neq 0$ does not contain
any algebraic curve.
Then we let $\pi _g : \mathcal{D}_g \to  S_g$ be a modular family of
smooth projective curves of genus $g$ as in \ref{p:modular}, 
and
\[
M'_g ({\cal X}) := \Hom 
({\cal D}_g \times \Delta / S_g \times \Delta, 
{\cal X} \times S_g / S_g \times \Delta)
\]
By \cite[Theorem II.1.7]{jK96}, we have
\begin{align*}
 \dim_{[\phi]} M'_g ({\cal X}) 
&\ge
\chi(\varphi^*(T_X))+\dim(S\times\Delta)\\
&= [-\deg(\varphi^*K_X)+2 \chi(\O_{\bar{C}})]
+ [3g-2+\dim (\Aut D)] \\
% &= 2
% \chi(\O_{\bar{C}})+[-3\chi(\O_{\bar{C}})+\dim(\mathrm{Aut}(\bar{C}))+1]\\
&= g+\dim (\Aut D).
\end{align*}
By construction and functoriality, an étale cover of $M'_g ({\cal X})$
maps finite-to-one into $M_g(X)$, so the above inequality implies the
required \eqref{ineq:required}.
\end{proof}

Note that Proposition \ref{P:estim} does not imply that the
varieties $V_{L,g}$ are non-empty. If the pair $(X,L)$ is general,
this is true for $0 \leq g \leq p_a$, as a consequence of the main
theorem in \cite{chen99}.
% (see e.g. \cite{th08}).

\begin{proposition}\label{T:ACZ-K3}
% Let $(X,L)$ be a polarized $K3$ surface, with $L^2=2p_a-2$,
% $p_a\ge 2$. Then for $0<g \le p_a$, 
For $g >0$,
the general element $C$ of
every irreducible component of $\V_{L,g}$ is immersed.
If moreover $C$ has a non-trigonal normalization,
then it is nodal.
\end{proposition}

\begin{proof}
We have
$\omega_{\bar{C}}\otimes\varphi^*\omega_X^{-1}=\omega_{\bar{C}}$.
This line bundle is globally generated since $g \ge 1$, and
$h^0(\omega_{\bar{C}}\otimes\varphi^*\omega_X^{-1})=g=\dim(V)$ by
Proposition~\ref{P:estim}. Therefore conditions (a) and (b) of
Theorem~\ref{T:nocusp} are satisfied and the first part follows.
If the normalization $\bar C$ is not trigonal, then condition (c)
of Theorem~\ref{T:nocusp} is also satisfied and $C$ is nodal.
\end{proof}

\begin{corollary}
\label{coro:CK-nontrigonal}
Let $(X,L)$ be a very general primitively
polarized $K3$ surface (\ie $L$ is indivisible in $\Pic X$)
with $L^2=2p-2$, $p \geq 2$, and $0 < g \le p$.
If
\begin{equation}
\label{ineq:CK}
g+ \Bigl\lfloor \frac{g}{4} \Bigr\rfloor
\big(g-2 \Bigl\lfloor \frac{g}{4} \Bigr\rfloor -2 \big)
> p,
\end{equation}
then the general element of
every irreducible component of $\V_{L,g}$ is nodal.
\end{corollary}

\begin{proof}
By \cite[Thm~3.1]{CK14}, inequality \eqref{ineq:CK}
ensures that for every $C \in |L|$,
the normalization of $C$ does not carry any $g^1_3$.
\end{proof}

\begin{proposition}
\label{P:K3-gnlzddivs}
Let $(X,L)$ be a very general primitively
polarized $K3$ surface, with $L^2 = 2p-2$.
If $g > \frac{p}{2}$, 
then the general element of
every irreducible component of $\V_{L,g}$ is nodal.
\end{proposition}

%%% paragraph below replaced by much more concise version

% \paragraph{}
% Before we start, recall that the 
% \emph{Clifford index} of a smooth curve $C$ of genus 
% $g \geq 2$ is
% \begin{equation*}
% \cliff(C):= \min \biggl\lbrace
% \Bigl\lfloor \frac{g-1}{2} \Bigr\rfloor,\
% \min\nolimits _{L \in \Pic C\ \text{s.t.}\  
% r(L)\geq 1\ \text{and}\ r(K_C-L)\geq 1}
% \bigl( \deg L -2r(L) \bigr)
% \biggr\rbrace,
% \end{equation*}
% where $r(M)$ stands for $\h^0(M)-1$ for any $M \in \Pic C$.
% Clifford's theorem states that $\cliff(C) \geq 0$,
% with equality holding if and only if $C$ is hyperelliptic.
% The Clifford index measures specialty with respect to 
% Brill--Noether theory: for instance, 
% it is $1$ if and only if it is either trigonal or a plane quintic,
% and $\bigl\lfloor \frac{g-1}{2} \bigr\rfloor$ for general $C$.

% Note moreover that as soon as $g \geq 4$, 
% there always exist some $L \in \Pic C$ such that
% $r(L)\geq 1$, $r(K_C-L)\geq 1$ and
% $\deg L -2r(L) \leq \bigl\lfloor \frac{g-1}{2} \bigr\rfloor$.

% Similarly, when $C$ is an integral curve of arithmetic genus 
% $p \geq 2$, we let
% \begin{equation*}
% \cliff(C):= \min \biggl\lbrace
% \Bigl\lfloor \frac{p-1}{2} \Bigr\rfloor,\
% \min\nolimits _{A \in 
% \operatorname{\overline{Pic}} C\ \text{s.t.}\  
% r(A)\geq 1\ \text{and}\ r(K_C-A)\geq 1}
% \bigl( \deg A -2r(A) \bigr)
% \biggr\rbrace,
% \end{equation*}
% with $\operatorname{\overline{Pic}} C$ the set of rank $1$ 
% torsion--free sheaves on $C$.

\paragraph{}
Before we prove this, recall that the 
\emph{Clifford index} of an integral projective curve $C$ of 
arithmetic genus $p \geq 2$ is
\begin{equation*}
\cliff(C):= \min \biggl\lbrace
\Bigl\lfloor \frac{p-1}{2} \Bigr\rfloor,\
\min\nolimits _{\bigl\{ A \in 
\operatorname{\overline{Pic}} C\ \text{s.t.}\  
r(A)\geq 1\ \text{and}\ r(K_C-A)\geq 1 \bigr\}}
\bigl( \deg A -2r(A) \bigr)
\biggr\rbrace,
\end{equation*}
where $\operatorname{\overline{Pic}} C$ is the set of rank $1$ 
torsion--free sheaves on $C$, and
$r(M)$ stands for $\h^0(M)-1$ for any $M \in 
\operatorname{\overline{Pic}} C$.
The bigger $\cliff(C)$ is, the more general $C$ is with respect to 
Brill--Noether theory.

\bigskip
\begin{proof}[Proof of Proposition~\ref{P:K3-gnlzddivs}]
We apply the strategy described in \ref{warn}, and 
circumvent the issue therein underlined by using the theory of
generalized divisors on singular curves, as recalled in 
subsection~\ref{s:gnlzddivs}
(we freely use the notations introduced in that subsection):
let $V$ be an irreducible component of $\V _{L,g}$,
$[C]$ a general member of $V$, and assume by contradiction that $C$ is
not nodal. 
We have
\[
\dim V \leq \dim \ES(C)
= \h^0 \bigl(C,I_C \otimes \O_C(C) \bigr),
\]
%and $I_C \subsetneq A_C$.
and we shall show that
\begin{equation}
%\tag{$\dagger$}
\label{ineq:AF}
% disugualianza assai famosa
\h^0 \bigl(C,I_C \otimes \O_C(C) \bigr)<
\h^0 \bigl(C,A_C \otimes \O_C(C) \bigr)=g,
\end{equation}
thus contradicting Proposition~\ref{P:estim} and ending the proof
(the right--hand--side equality in \eqref{ineq:AF} comes from 
Lemma~\ref{L:indep-adj}).

If $\h^1 \bigl(C,  \O_C (E) \bigr) <2$, then
\begin{equation*}
\h^0 \bigl( I \otimes \O_C(C) \bigr)
= \h^0 \bigl( \omega_C(-E) \bigr)
= \h^1 \bigl( \O_C (E) \bigr) \leq 1
< g,
\end{equation*}
and \eqref{ineq:AF} holds.
If on the other hand $\h^0 \bigl(C,  \O_C (E) \bigr) <2$, then
\eqref{ineq:AF} still holds, since Lemma~\ref{L:notnodal2} together
with \ref{p:rigid} yield
\begin{equation*}
\h^0 \bigl(C, N _{C/X} \otimes A) \bigr)
- \h^0 \bigl(C, N _{C/X} \otimes I \bigr)
> 1 - \h^0 \bigl(C,  \O_C (E) \bigr).
\end{equation*}
For the remaining of the proof, we therefore assume that both
$\h^0 \bigl(C,  \O_C (E) \bigr)$ and
$\h^1 \bigl(C,  \O_C (E) \bigr)$ are $\geq 2$.

Now, being $(X,L)$ a very general primitively polarized $K3$ surface,
and $C \in |L|$ an integral curve of geometric genus $g \geq 2$,
it follows from \cite{BFT} together with \cite{GL} that the Clifford
index of $C$ is that of a general smooth curve of genus $p$,
\ie $\cliff(C)= \lfloor \frac{p-1}{2} \rfloor$.
This implies
\begin{equation*}
p+1 - \bigl[
\h^0\bigl( \O_C(E) \bigr) + \h^0\bigl( \omega_C(-E) \bigr)
\bigr] 
= \deg E -2r(E) 
\geq \Bigl\lfloor \frac{p-1}{2} \Bigr\rfloor,
\end{equation*}
hence
\[
\h^0\bigl( \omega_C(-E) \bigr)
\leq \frac{p}{2} +2 -\h^0\bigl( \O_C (E) \bigr)
\leq \frac{p}{2},
\]
so that \eqref{ineq:AF} again holds.
\end{proof}

\begin{remark}
In a private correspondence concerning a previous version of this
paper, 
X.~Chen has shown (using methods completely different from ours) 
that the statement of Proposition~\ref{P:K3-gnlzddivs} holds
more generally without the limitation $g > \frac{p}{2}$.
\end{remark}

\begin{remark}
\label{R:K3} 
The case $g=0$ in Proposition~\ref{T:ACZ-K3} is a true exception. For
example, there exist irreducible rational plane quartic curves
with one cusp and two nodes. Pick a general such curve: then there
is a nonsingular quartic surface in $\P^3$ containing it as a
hyperplane section.
\end{remark}

On the other hand, it seems fairly reasonable to formulate the
following conjecture, which predicts that the case $g=0$ holds for
very general $(X,L)$. It is of particular interest in the context
of enumerative geometry, in that it provides a good understanding
of the various formulae counting rational curves on $K3$ surfaces
(see \cite{FGS99,beauville-counting}).
\begin{conjecture}
Let $(X,L)$ be a very general polarized $K3$ surface. Then all
rational curves in $|L|$ are nodal.
\end{conjecture}

This has been proved by Chen \cite{chen} in the particular case
of indivisible $L$, using a degeneration argument.

\subsubsection{Enriques surfaces}

\begin{theorem}\label{C:ACZ2}
Let $X$ be an Enriques surface and $L$ an invertible sheaf on $X$.
If $g \ge 3$, and $[C]\in \V_{L,g}$ has a non-hyperelliptic
normalization $\bar{C}$, then the general element of every
component of $\V_{L,g}$ containing $C$ has no cusps. If moreover
$\cliff(\bar C) \ge 5$ then $C$ is nodal.
\end{theorem}

\begin{proof}
The sheaf $L:=\omega_{\bar{C}}\otimes\varphi^*\omega_X^{-1}$ has
degree $2g-2$, and  it is  Prym-canonical: in particular, it is
non-special.  On a  non-hyperelliptic curve, every Prym-canonical
sheaf is globally generated  (see, e.g., \cite[Lemma (2.1)]{LS96}).
Therefore, the first part follows from Corollary \ref{C:nocusp1}.
If $\cliff(C) \ge 5$ then $\varphi_L(C)\subset \P^{g-2}$
has no trisecants, by \cite[Proposition~(2.2)]{LS96}, and
therefore condition \ref{c:3ample} of Theorem \ref{T:nocusp} is also
satisfied.
\end{proof}

\subsubsection{Abelian surfaces}

Let $(X,\xi)$ be a polarized Abelian surface, and let
$p=p_a(\xi)$. For each $[C] \in \curves^\xi_X$ we have
$\dim |C|=p-2$, so that   $\curves^\xi_X$ is a
$\P^{p-2}$-fibration over the dual Abelian surface
$\widehat{X}$. A general Abelian surface does not contain any
curve of geometric genus $\leq 1$. 
On the other hand, the arguments for
Propositions~\ref{P:estim} and \ref{T:ACZ-K3} apply
mutatis mutandis to this situation, so one has:

\begin{proposition}
Let $2 \leq g \leq p_a$, and $V$ an irreducible component of
$V^\xi_g$.
Then $\dim V = g$, and the general $[C] \in V$ corresponds to a curve
with only immersed singularities.
If moreover $C$ has non--trigonal normalization, then it is nodal.
\end{proposition}

Note however that, unlike the case of $K3$ surfaces, 
we do not know in general whether the varieties $V^\xi_g$
are non-empty for  $2 \leq g \leq p_a$.

For genus $2$ curves, more is known \cite[Prop.~2.2]{LS02}:
if $(X,L)$ is an Abelian surface of type $(d_1,d_2)$, then any
genus $2$ curve in $|L|$ 
has at most ordinary singularities of multiplicity $\leq
\frac{1}{2}(1+\sqrt{8d_1d_2-7})$. 
We have the following enlightening and apparently
well--known example which, among other things, shows
that this bound is sharp.

\begin{example}
\label{Ex:jacobian} Let $X$ be the Jacobian of a general genus $2$
curve $\Sigma$, and choose an isomorphism 
$X \simeq \Pic^1\Sigma$;
it yields an identification $\Sigma \simeq \Theta_\Sigma$. Denote by
$\{\Theta_\Sigma\}$ the corresponding polarization on $X$. The curve
$\Sigma$ has six Weierstrass points $w_1,\ldots,w_6$, and the divisors
$2w_i$ on $\Sigma$ are all linearly equivalent. It follows that the
image of $\Sigma \subset X$ by multiplication by $2$ is an irreducible
genus $2$ curve $C$ which belongs to the
linear system $\bigl|2^2\cdot \Theta_\Sigma \bigr|$,
and has a $6$-fold point, the latter being ordinary by
\cite[Prop.~2.2]{LS02} quoted above.

The curve $C$ and its translates are parametrized by an irreducible
(two-dimensional) component $V$ of $V^{\{4\Theta_\Sigma\}}_2$. Since
$\omega_{\bar{C}}\otimes\varphi^*\omega_X^{-1}=\omega_{\Sigma}$ is
globally generated and $\dim V=2=\h^0(\omega_\Sigma)$,
conditions (a) and (b) of Theorem \ref{T:nocusp} are satisfied.
 On the other
hand, condition (c) of Theorem \ref{T:nocusp} is clearly not
fullfilled and   $C$ is not nodal,   showing that this condition
  is not redundant.

We emphasize that this is an explicit illustration of the warning
given in \ref{warn}.
We have here (letting as usual $\nu$ denote the normalization of $C$)
\[
\nu^* \bigl| N _{C/X} \otimes A_C \bigr| 
= \bigl| \nu^* (N _{C/X} \otimes A_C) \bigr| 
= |\omega _{\bar C} |
\]
which is a base--point--free linear system on $\bar C$, 
whereas
\[
\bigl| N _{C/X} \otimes A_C \bigr| 
= \bigl| N _{C/X} \otimes I_C \bigr| 
\]
even though $I_C \subsetneq A_c$.
Observe also that $\nu^* I_C = \nu^* A_C$ by Example~\ref{ex:m-uple}.
\end{example}

\section{A museum of noteworthy behaviours}
\label{S:museum}

\subsection{Maximal equigeneric families with non-nodal
general member} 
\label{s:counterex}

The examples in this subsection are mainly intended to show that the
assumption that $\omega_{\bar{C}}\otimes \varphi^*\omega_X^{-1}$
is globally generated in Theorem~\ref{T:nocusp} is necessary.
The same goal was achieved by the examples provided in Remarks
\ref{R:delpezzo} and \ref{R:K3}, but the ones presented here are
hopefully less peculiar (e.g., the involved equigeneric families are 
in general not $0$-dimensional).

\begin{example}
\label{Ex:cayleybacharach}
{\em (a complete positive dimensional
ample linear system on a rational surface, all
  members of which have a cuspidal double point)}

The surface will be a plane blown-up at distinct points, which
will allow us the use of a Cayley-Bacharach type of argument.
Let $C_1,C_2 \subset \P^2$ be two irreducible sextics having an
ordinary cusp at the same point $s_0 \in \P^2$, with the same
principal tangent line, no other singularity, and meeting
transversely elsewhere.
Their local intersection number at $s_0$ is $(C_1\cdot C_2) _{s_0}=6$,
so we can consider $26$
pairwise distinct transverse intersection points $p_1, \dots,
p_{26} \in C_1 \cap C_2 \setminus \{s_0\}.$ Let $\pi:X \to \P^2$ be
the blow-up at $p_1, \dots, p_{26}$, and let 
$L := 6H- \sum_{1 \leq i \leq 26} E_i$, where
 $H=\pi^*\O_{\P^2}(1)$, and the $E_i$'s are the exceptional curves
of $\pi$.
 Then, since $\dim|\O_{\P^2}(6)|=27$, $|L|$ is a pencil generated by
 the proper transforms of $C_1$ and $C_2$, hence consists entirely
 of curves singular at the point ${s}= \pi^{-1}(s_0)$
and with a non ordinary singularity there. The general $C \in
|L|$ is irreducible of genus
 nine, and $\V_{L,9}$ is therefore an open subset of $|L|$, not
 containing any nodal curve.

For general $C \in |L|$, one computes
$\h^0(\omega_{\bar{C}}\otimes \varphi^*\omega_X^{-1})=1$,
which shows that the line bundle
$\omega_{\bar C}\otimes \varphi^* \omega_X^{-1}$ on $\bar C$
is not globally generated
(we let, as usual, $\bar C \to C$ be the normalization of $C$,
and $\varphi$ its composition with the inclusion $C \subset X$).
Thus condition~\ref{c:bpf} of Theorem~\ref{T:nocusp} does not hold,
while condition~\ref{c:dim} %of Theorem \ref{T:nocusp}
is verified.
As a sideremark, note that $(-K_X\cdot L)<0$ %=-8
and $L$ is ample
(see also Remark~\ref{Ex:nonredM} above about this example).

This example can be generalized to curves with an arbitrary number of
arbitrarily nasty singularities:
% (including one that is supported at several points):
simply note that the dimension of $|\O_{\P^2}(d)|$ grows as
$d^2/2$ when $d$ tends to infinity, and is therefore smaller by as
much as we want than the intersection number of two degree $d$
plane curves for $d$ big enough.
\end{example}

\begin{examples}
The forthcoming examples all are degree $n$ cyclic coverings
$\pi:X \to \P^2$, branched over a smooth curve $B \subset \P^2$
of degree $d$.
They are smooth and regular.
% The latter point can be checked as follows:
% consider a line bundle $\cal L$ on $\P^2$ such that
% $\O_{\P^2}(B)={\cal L}^{\otimes d}$; then
% $\pi_*\O_X=\bigoplus_{k=0}^{-d+1} {\cal L}^{\otimes k}$, and since on
% the other hand one has $R^i\pi_* \O_X=0$ for all $i \geq 1$, one
% computes
% \[
% {\rm H}^1(X,\O_X)=
% \bigoplus_{k=0}^{-d+1} {\rm H}^1(\P^2,{\cal L}^{\otimes k})=0
% \]
% using Leray's spectral sequence.
Let $L= \pi^* \O_{\P^2} (1)$. One has
\[
\H^0 \bigl(X, kL \bigr)
= \pi^* \H^0 \bigl(\P^2, \O _{\P^2} (k) \bigr)
\]
if and only if $k < \frac{d}{n}$.

\medskip 
\subparagraph{}
{\em (a complete ample linear system with a
codimension $1$
  equigeneric stratum, the general member of which has an
  $A_{n-1}$-double point)}

As a local computation shows, the inverse
image in $X$ of a plane curve simply tangent to $B$ is a curve with an
$A_{n-1}$--double point at the preimage of the tangency point.

It follows that for $1 \leq k < \frac{d}{n}$
there is a codimension $1$ locus in $|kL|$ that parametrizes curves
with an $A_{n-1}$-double point,
% (namely, the pull--back of the hypersurface in 
% $|\O_{\P^2} (k)|$ parametrizing curves tangent to $B$),
although the general member of
$|kL|$ is a smooth curve. This is an irreducible component of
$V_{kL,p_a(kL)-\lfloor d/2 \rfloor}$. 
It is superabundant,
since one expects in general that codimension $c$ equigeneric
strata are components of $V_{kL,p_a(kL)-c}$.  

\medskip 
\subparagraph{}
{\em (a complete ample linear system containing two
codimension $1$
  equigeneric strata, that respectively parametrize curves of
  genera $g_1$ and $g_2$, $g_1 \neq g_2$)}

The inverse image in
$X$ of a plane curve having a node outside of $B$ is a curve
having $n$ distinct nodes. Consequently, there is for every 
$3 \leq k < \frac{d}{n}$
a codimension $1$ locus in $|kL|$ that parametrizes
integral curves with $d$ distinct nodes. 
This is an irreducible component
of $V_{kL,p_a(kL)-d}$, and it too is superabundant.

As a conclusion, notice that the discriminant locus in $|kL|$ is
reducible, and has two of its irreducible components contained in
$V_{kL,p_a(kL)-\left[d/2\right]}$ and $V_{kL,p_a(kL)-d}$
respectively.

\medskip 
\subparagraph{}
{\em (further examples of $0$-dimensional equigeneric loci)}

\begin{inparaenum}[(i)]
\item Assume there exists a line which meets $B$ at some point $s$
with multiplicity $4$. Then its inverse image in $X$ is a curve
with a tacnode. The corresponding point of $|L|$ is a component of
$V_{L,p_a(L)-2}$, that is not superabundant.

\item Assume there exists a line $D$ which is tangent to $B$ at
three
  distinct points: its inverse image is then a curve with three
  distinct nodes.
The corresponding point of $|L|$ is a component of
$V_{L,p_a(L)-3}$: it does correspond to a nodal curve, but it is
superabundant. If one further assumes one of the tangency points
of $D$ with $B$ to be a flex of $B$, then the inverse image of $D$
is no longer nodal: it has two nodes and a cusp.
\end{inparaenum}

It should be clear by now, how these two examples can be
generalized to produce an infinite series of examples.
\end{examples}

\subsection{Singular maximal equisingular families}

Let $X$ be a smooth projective surface, $\xi \in \NS (X)$,
and $C$ an integral curve of genus $g$ and class $\xi$.
We wish to illustrate in this subsection the fact that the
local structures at $[C]$ of both $V^\xi _{g}$ and $\ES(C)$ are not as
nice as one would expect them to be by looking at their counterparts
in the deformation theory of a single planar curve singularity.
In fact, the situation is already messy in the simplest case
$X=\P^2$.

\paragraph{}
Let $p_1, \dots, p_\delta $ be the singular points of $C$,
and let $\hat C_i$ be the germ of $C$ at $p_i$ for
each $i=1, \dots, \delta$, and
\[
\xymatrix@=15pt{\hat C_i \ar[d]\ar@{^(->}[r] & \hat{\cC}_i \ar[d] \\
\mathrm{Spec}(\mathbf{C}) \ar[r]& B_i}
\]
be the \'etale semiuniversal deformation of $\hat{C}_i$
(see \cite{DH88} for a precise account on this).
By their universal properties, there exists an \'etale neighborhood
$W \to \curves ^\xi _X$ of $[C]$ such that there is a restriction morphism
\begin{equation*}
%\label{E:restr}
\textstyle{r: W \to \prod_i B_i.}
\end{equation*}
The general philosophy we want to underline 
can be summed up as follows.

\begin{remark}
\label{R:nonsmooth}
{\em In general, the restriction map $r$ is not smooth.}
\end{remark}
Note that both domain and codomain of $r$ are smooth.
In particular, the smoothness of $r$ is equivalent to the surjectivity
of its differential.

% Recall that the tangent space of $\prod_i B_i$ at its closed point is
% canonically identified with ${\rm H}^0({\rm T}^1_C)$, where  ${\rm
%   T}^1_C$ is the first cotangent sheaf of $C$, which is defined as
% \[
% {\rm T}^1_C := \mathrm{coker}\left(\partial:T_X \to
% \O_C(C)=N_C\right).
% \]
% Letting $N'_C= \mathrm{Im}(\partial)$, we have an exact sequence
% \[
% 0 \to N'_C \to  N_C \to  {\rm T}^1_C \to 0,
% \]
%  and the differential $dr_{[C]}$ is the induced map ${\rm H}^0(N_C)
%  \to {\rm H}^0({\rm T}_C^1)$.  Therefore, the
%  smoothness of $r$ is equivalent to the injectivity of
%  \[
% {\rm H}^1(N'_C) \to {\rm H}^1(N_C).
%  \]

\paragraph{}
The equigeneric and equisingular loci inside each one of
the deformation spaces $B_i$ are known to be well-behaved 
(we refer to \cite{DH88} for details).
Among others, let us mention that the equisingular locus is
smooth, and that the general point in the equigeneric locus
corresponds to a deformation of $p_i$ in a union of nodes.
Now, the smoothness of $r$ would transport these good properties to
$\V ^\xi _{g}$ and $\ES(C)$.
In particular, it would imply the two following facts:
\begin{compactenum}
\renewcommand{\theenumi}{\normalfont (\alph{enumi})}
\item
\label{csq:ACZ}
the general point of every irreducible component $V$ of
$\V ^\xi _{g}$ corresponds to a nodal curve;
\item
\label{csq:ES}
$\ES(C)$ is smooth, and of the expected codimension in
$\curves ^\xi _X$.
\end{compactenum}

Now Remark \ref{R:nonsmooth} follows from the fact that neither
\ref{csq:ACZ} nor \ref{csq:ES} is true in general.
For \ref{csq:ACZ}, this was discussed previously in \S
\ref{s:counterex}.
On the other hand, property \ref{csq:ES} can
be contradicted in several ways:
we refer to \cite{GLS} for a discussion of these problems and
for a survey of what is known. Here we solely mention a few
examples which are relevant to our point of view.

\begin{examples}
If $C$ has $n$ nodes, $\kappa$ ordinary cusps, and no
further singularity, then $\ES (C)$ is the locus of curves 
with $n$ nodes and $\kappa$ cusps, and has expected codimension
$n+2\kappa$ in $\curves ^\xi _X$.
Here, we let $X=\P^2$, and adopt
the usual notation $V_{d,n,\kappa}$ for the scheme of irreducible
plane curves of degree $d$, with $n$ nodes, $\kappa$ cusps, and no
further singularity.

\medskip 
\subparagraph{}
(B. Segre \cite{bS29I}, see also \cite[p. 220]{oZ71})
{\em For $m \ge3$, there exists an irreducible component of
 $\V_{6m,0,6m^2}$,
 which is nonsingular and has dimension strictly larger than
the expected one.}

\medskip 
\subparagraph{}
(Wahl \cite{jW74b})
{\em The scheme $\V_{104,3636,900}$
has a non-reduced component of dimension $174 > 128= \frac{104\cdot
  107}{2}-3636-2 \cdot 900$.}

\medskip 
\subparagraph{}
{\em There also exists an equisingular stratum $V_{d,n,\kappa}$ having a
reducible connected component.}

\smallskip
The construction of the latter, which we shall now outline,
follows the same lines as that of Wahl \cite{jW74b},
and is based on the example of \cite{eS81}
(for a thorough description of which we refer to
\cite[\S 13 Exercises]{hart}).

Start from a nonsingular curve $A$ of type  $(2,3)$ on a nonsingular
quadric $Q \subset \P^3$, and let $F,G \subset \P^3$ be respectively a
general quartic and a general sextic containing $A$.
Then $F\cap G = A \cup \gamma$ where $\gamma$ is   a nonsingular curve
of degree 18 and genus 39.
As shown in \cite{eS81}, the curve $\gamma$ is obstructed. Precisely,
$[\gamma]$ is in the closure of two components of
$\mathrm{Hilb}^{\P^3}$, each consisting generically of projectively
normal, hence unobstructed, curves.

Now consider an irreducible surface $S \subset \P^3$ of degree $N
\gg 0$, having ordinary singularities along $\gamma$, and let $C
\subset \P^2$ be the branch curve of a generic projection of $S$
on $\P^2$, $d:= \deg(C)$. By \cite{CF08}, $C$ is irreducible,
and has $n$ nodes and $\kappa$ cusps as its only singularities. It
then follows from the results of \cite{jW74b}, that
$\mathrm{Hilb}^{\P^3}$ at $[\gamma]$ is smoothly related with
$\V_{d,n,\kappa}$ at $[C]$.  Therefore  $\V_{d,n,\kappa}$ is
analytically reducible at $[C]$.

In fact, one can show more precisely that $\V_{d,n,\kappa}$ is
reducible at $[C]$, by taking generic projections of irreducible
surfaces $S'$ of degree $N$ having ordinary singularities along curves
$\gamma'$ which are in a neighbourhood of $[\gamma]\in
\mathrm{Hilb}^{\P^3}$.
\end{examples}

\begin{closing}
% \bibliographystyle{mysmf-alpha}
% \bibliography{equigeneric}

% biblio ci--dessous obtenue avec BibTeX et les lignes ci--dessus

\bigskip
\noindent\textsc{%
Institut de Math\'ematiques de Toulouse (CNRS UMR 5219),
Universit\'e Paul Sabatier,
31062 Toulouse Cedex 9, France} \\
\texttt{thomas.dedieu@m4x.org}

\bigskip
\noindent\textsc{%
Dipartimento di Matematica e Fisica,
Universit\`a Roma Tre,
L.go S.L.\ Murialdo 1,
00146 Roma, Italy} \\
\texttt{sernesi@mat.uniroma3.it}
\end{closing}

\end{document}